\documentclass[11pt]{amsart}
\usepackage{amssymb,latexsym}
\usepackage[utf8]{inputenc}
\usepackage{amsmath,amsthm}
\usepackage{amsfonts,mathrsfs}

\usepackage[margin=1in]{geometry}

\usepackage{enumerate,units}
\usepackage[all]{xy}
\usepackage{graphicx}

\usepackage{hyperref, url}
   
\theoremstyle{plain}
\newtheorem{theorem}{Theorem}[section]
\newtheorem{corollary}[theorem]{Corollary}
\newtheorem{lemma}[theorem]{Lemma}
\newtheorem{proposition}[theorem]{Proposition}

\newtheorem{fact}[theorem]{Fact}
\newtheorem*{claim}{Claim}
\newtheorem*{theorem*}{Theorem}
\newtheorem*{proposition*}{Proposition}

\theoremstyle{definition}
\newtheorem{definition}[theorem]{Definition}
\newtheorem{example}[theorem]{Example}

\theoremstyle{remark}
\newtheorem{remark}[theorem]{Remark}
\newtheorem{notation}[theorem]{Notation}
\newtheorem{question}[theorem]{\textbf{Question}}

\numberwithin{equation}{section}
\newcommand{\forkindep}[1][]{%
  \mathrel{
    \mathop{
      \vcenter{
        \hbox{\oalign{\noalign{\kern-.3ex}\hfil$\vert$\hfil\cr
              \noalign{\kern-.7ex}
              $\smile$\cr\noalign{\kern-.3ex}}}
      }
    }\displaylimits_{#1}
  }
}

\newenvironment{claimproof}[1][\proofname]
  {%
    \proof[#1]%
  }
  {%
    \endproof%
  }

\newcounter{step}                   
    {\hfill $\clubsuit$             
     \vspace{7pt}\par}

\newcommand{\aff}{\text{aff}}
\newcommand{\GVar}[1][K]{(\text{SPVar}/{#1})}

\newcommand{\aGVar}[1][K]{(\text{SPVar}^{\aff}/{#1})}

\newcommand{\cV}{\mathcal{V}}
\newcommand{\cG}{\mathcal{G}}
\newcommand{\cE}{\mathcal{E}}

\newcommand{\cA}{\mathcal{A}}
\newcommand{\cU}{\mathcal{U}}
\newcommand{\cW}{\mathcal{W}}
\newcommand{\cO}{\mathcal{O}}
\newcommand{\cX}{\mathcal{X}}
\newcommand{\cY}{\mathcal{Y}}

\newcommand{\cS}{\mathcal{S}}

\newcommand{\m}{\mathfrak{m}}
\newcommand{\bk}{\textbf{k}}
\DeclareMathOperator{\Spec}{Spec}

\DeclareMathOperator{\val}{val}
\DeclareMathOperator{\res}{res}

\DeclareMathOperator{\acl}{acl}
\DeclareMathOperator{\dcl}{dcl}
\DeclareMathOperator{\tp}{tp}

\begin{document}
\title{Models of Abelian varieties over valued fields, using model theory}
\author{Yatir Halevi}
\address{Department of Mathematics\\ University of Haifa\\ 199 Abba Khoushy Avenue \\ Haifa \\Israel}
 \email{ybenarih@campus.haifa.ac.il}

\thanks{The author was partially supported by ISF grants No. 555/21 and 290/19 and by the Fields Institute for Research in Mathematical Sciences.}

\begin{abstract}
Given an elliptic curve $E$ over a perfect defectless henselian valued field $(F,\mathrm{val})$ with perfect residue field $\textbf{k}_F$ and valuation ring $\mathcal{O}_F$, there exists an integral separated smooth group scheme $\mathcal{E}$ over $\mathcal{O}_F$ with $\mathcal{E}\times_{\text{Spec } \mathcal{O}_F}\text{Spec } F\cong E$. If $\mathrm{char}(\bk_F)\neq 2,3$ then one can be found over $\mathcal{O}_{F^{alg}}$ such that the definable group $\mathcal{E}(\mathcal{O})$ is the maximal generically stable subgroup of $E$. We also give some partial results on general Abelian varieties over $F$.

The construction of $\mathcal{E}$ is by means of generating a birational group law over $\mathcal{O}_F$ by the aid of a generically stable generic type of a definable subgroup of $E$.
\end{abstract}

\maketitle

\section{Introduction}
Let $E$ be an elliptic curve over the field of fractions $F=\mathrm{Frac}(R)$ of a valuation ring $R$.  A \emph{model} of $E$ is a scheme $\cE$ (of finite type) over $R$ for which $\cE\times_{\Spec R} \Spec F\cong E$. There are many different models one can consider. For example, if $E$ is given by a Weierstrass equation over $R$ then we can take the projective closure in $\mathbb{P}^2_R$; we get a proper model over $R$ but it might not be smooth. If we now take the smooth locus over $R$ the resulting model is smooth but we might lose  properness. We might also desire the model to be a group scheme over $R$.

When $R$ is a discrete valuation ring, for any abelian variety $A$ over $F$, N\'eron constructed a model $\cA$ over $R$ which can be seen as a best possible model. N\'eron relaxed the assumption of properness and concentrated on smoothness and the group structure. The model N\'eron constructed (\emph{the N\'eron model of $A$ over $R$}) is a smooth separated group scheme of finite type over $R$ satisfying that $\cA\times_{\Spec R} \Spec F\cong A$ and the \emph{N\'eron mapping property}. This latter property is extremely important, but will not be important here; we just remark that it implies that $\cA(R)=\cA_F(F)$.

It is well known that if $R$ is a valuation subring of an algebraically closed field then a N\'eron model might not necessarily exist, even for an elliptic curve over $R$ (see, e.g., Remark \ref{R: neron model of elliptic over acvf not exists}). Nevertheless, one might hope to find well-behaved models. See  \cite{neron,bosch} for more information on models over DVRs.

The main aim of this paper is to give a model theoretic approach for constructing such models by means of finding  birational group laws (in the sense of \cite{bosch}) over $R$. We build heavily on results from \cite{Metastable} and \cite{stab-pointed-published}, where generically stable groups in algebraically closed valued fields and stably pointed varieties were studied, respectively. Before stating the main theorems we elaborate on the former.

Let $(K,\val)$ be an algebraically closed valued field. A type-definable subgroup $H\subseteq G$ of an algebraic group $G$ over $K$ is said to be a connected generically stable group if there exists a generically stable type $p$, concentrated on $H$, satisfying that $gp=p$ for any $g\in H$ (so $p$ is the unique generic type of $H$). We recall that a global $K$-invariant type $p$ concentrated on $G$ is generically stable if for any open affine subvariety $V\subseteq G$, on which $p$ is concentrated and $L\succ K$, $\val(f(c))\in \Gamma_L$ for any $f\in L[V]$ and $c\models p|L$.

We now give the main theorem for elliptic curves.  Let $F$ be a perfect defectless henselian (non-trivially) valued field with a perfect residue field $\bk_F$. Let $\cO_F$ be its valuation ring  and let $\cO_{F^{alg}}$ be its extension to the algebraic closure $F^{alg}$ of $F$. Below the types and definable groups are objects in the algebraically closed valued field $F^{alg}$.

\begin{theorem*}[Theorem \ref{T:good reduction for elliptic curves}, Theorem \ref{T:geo interp}]
Let $E$ be an elliptic curve given by a Weierstrass equation over $\cO_F$.
\begin{enumerate}

\item  There exists a smooth integral separated $\cO_F$-group scheme $\cE$ of finite type over $\cO_F$ with geometrically integral fibers satisfying that $\cE_F\cong E$ and
\begin{enumerate}
\item $\cE(\cO_{F^{alg}})$ is a connected generically stable subgroup of $\cE(F^{alg})=E(F^{alg})$,
\item $\cE(\cO_{F^{alg}})=E(F^{alg})$ if and only if $\cE$ is an Abelian scheme over $\cO_F$ if and only if $\Delta\in \cO_F^\times$, where $\Delta$ is the discriminant of $E$.

\end{enumerate}

\item Assuming $\mathrm{char}(\bk_F)\neq 2,3$, there exists $\gamma_\infty\in \Gamma_{F^{alg}}$ satisfying:
\begin{enumerate}
	\item For any $\gamma\leq \gamma_\infty$ there is a smooth integral separated $\cO_{F^{alg}}$-group scheme $\cE_{\gamma}$ of finite type over $\cO_{F^{alg}}$ with integral fibers and $(\cE_{\gamma})_{F^{alg}}\cong E_{F^{alg}}$.
	\item  For any $\gamma_1< \gamma_2\leq \gamma_\infty$, $\cE_{\gamma_1}(\cO_{F^{alg}})\subsetneq \cE_{\gamma_2}(\cO_{F^{alg}})$ and both are generically stable subgroups of $E(F^{alg})$.
	\item Conversely, for any integral group scheme $\cG$ of finite type $\cO_{F^{alg}}$ with integral special fiber and $\cG_{F^{alg}}\cong E_{F^{alg}}$, there exists $\gamma\leq \gamma_\infty$ with $\cG(\cO_{F^{alg}})=\cE_\gamma(\cO_{F^{alg}})$, and
	\item For any connected generically stable type-definable subgroup $H\leq G$ there is $\gamma\leq \gamma_\infty$ with $H=\cE_\gamma(\cO_{F^{alg}})$.
\end{enumerate}
\end{enumerate}
\end{theorem*}

If $F$ has discrete value group and $E$ is given by a minimal Weierstrass equation then $\cE$ is isomorphic to the identity component of the N\'eron model of $E$ (Proposition \ref{P:neron}).

For general Abelian varieties over an algebraically closed valued field $K$, with valuation ring $\cO_K$, we have the following partial result.

\begin{theorem*}[Theorem \ref{T:Abelian varieties}]
Let $A$ be an Abelian variety over $K$ and assume that $A$ is a connected generically stable group with a unique generic type $p$. Then the following are equivalent:
\begin{enumerate}
\item There exists an open affine subvariety $V\subseteq A$ such that \[\{f\in K[V]:p\vdash \val(f(x))\geq 0\}\] is a finitely generated algebra over $\cO_K$.
\item There exists an integral Abelian scheme $\cA$ over $\cO_K$ with $\cA_K\cong A$.
\end{enumerate}
\end{theorem*} 

%
 Our hope is that this theorem may be generalized to the generic type of the maximal generically stable subgroup, in the sense of \cite{Metastable}, and that condition (1) always holds.

The key model theoretic input is the following, stated for algebraically closed valued fields just for simplicity; it holds for valued fields as above but requires some more assumptions. It should be seen as an analog of \cite[Theorem 6.11]{Metastable} for non-affine algebraic groups. We do not know if the finite generation assumption can be dropped.

\begin{proposition*}[Proposition \ref{P: existence of a group scheme}]
Let $G$ be an algebraic group over $K$ and let $H$ be a connected generically stable Zariski dense type-definable subgroup with $p$ its generically stable generic.  Further assume that there exists an open affine subvariety $V\subseteq G$ with \[\{f\in K[V]:p\vdash \val(f(x))\geq 0\}\] a finitely generated algebra over $\cO_K$.

Then there exists a smooth integral separated $\cO_K$-group scheme $\cG$ of finite type over $\cO_K$ with geometrically integral fibers satisfying that $\cG_K\cong G$ and that under this isomorphism $H\cong \cG(\cO)$.
\end{proposition*}

\noindent{\bf Acknowledgements} I would like to thank Silvain Rideau-Kikuchi for numerous conversations and Martin Hils for going over an early draft. Appendix \ref{S:generically stable and shelah} is a result of several discussions with Itay Kaplan. I would also like to thank the anonymous referee for many useful comments and finding a gap in a previous draft. Finally, I would like to thank Udi Hrushovski, the seeds (and seedlings) of this paper grew out of our fruitful discussions during my PhD.

\section{Preliminaries and notation}
We will usually not distinguish between singletons and sequences, thus we may write $a\in M$ and actually mean $a=(a_1,\dots,a_n)\in M^n$, unless a distinction is necessary. We use juxtaposition $ab$ for concatenation of sequences, or $AB$ for $A\cup B$ if dealing with sets. That being said, since we will be dealing with groups, we will try to differentiate between concatenation $ab$ and group multiplication $ab$ by denoting the latter by $a\cdot b$. 

We review the necessary model theory of algebraically closed valued fields. See \cite{stab-pointed, Non-Arch-Tame, StableDomination} for more information. We will assume basic knowledge of model theory (and specifically the model theory of NIP structures), see \cite{guidetonip}.

We will also require many results from \cite{stab-pointed-published}. We will not be able to fully write all of them here as facts, but will cite them carefully when needed. The published version \cite{stab-pointed-published} has several minor mistakes and one major error, see \cite{stab-pointed-corr} for a full corrigendum.  The arxiv version \cite{stab-pointed} has all the corrections implemented so it will be our main cited version.

\subsection*{Valued fields} 
Let $(F,\val )$ be a non-trivially valued field with value group $\Gamma_F$. The valuation ring of $F$ is the ring $\mathcal{O}_F=\{x\in F:\val (x)\geq 0\}$. It is a local ring with maximal ideal $\m_F=\{x\in F:\val (x)>0\}$. The residue field is the quotient $\bk_F=\mathcal{O}_F/\m_F$ and the quotient map $\res :\mathcal{O}_F\to \bk_F$ is called the residue map. 

Let ACVF be the theory of non trivially valued algebraically closed  fields in the three sorted language: the valued field sort which will be denoted by $\mathrm{VF}$, the value group by $\Gamma$ and the residue field by $\bk$.

We will treat the valued field sort and the model as interchangeable, e.g. when we say that $K$ is a model of ACVF we really mean that $K=\mathrm{VF}(M)$ for some $M\models$ACVF. For any definable set $D$ and set $A$, we will write $D(A)$ for $\dcl(A)\cap D$. So
$\Gamma (A):= \dcl(A)\cap \Gamma$, and $\bk(A):= \dcl(A)\cap \bk$

We work in some large (larger than any set or field in question) saturated model $\mathbb{K}$ of ACVF (technically $\mathbb{K}=\mathrm{VF}(\mathbb{U})$ for some saturated model $\mathbb{U}$).\footnote{There are standard techniques from set theory that ensure the generalized continuum hypothesis from some point
on while fixing a fragment of the universe (so this does not affect the questions asked in this paper), see \cite{HaKa}.} All valued fields in question will be seen as  subfields of $\mathbb{K}$. We set $\cO=\cO_{\mathbb{K}}$, $\Gamma=\Gamma_{\mathbb{K}}$ and $\bk=\bk_{\mathbb{K}}$.

The theory ACVF has NIP;  we will mostly deal with \emph{generically stable} types in this theory. The precise definition of such types will not be important here (see Appendix \ref{S:generically stable and shelah}). What will be important is that in ACVF an invariant type is generically stable if and only if it is stably dominated if and only if it is orthogonal to $\Gamma$.

This latter definition will be our most used property of such types. We say that a global invariant type $p$ is \emph{orthogonal to $\Gamma$} (over $A$) if it is $A$-definable and for every $B\supseteq A$ and $B$-definable function $f$ into $\Gamma$, $f_*p$ is a constant type, where $f_*p=\tp(f(a)/\mathbb{U})$ for $a\models p$. For a global invariant type $p$ concentrated on an affine variety $V$ over a model $K\models\text{ACVF}$ this exactly means that $\val(f(c))\in \Gamma_L$ for any $L\succ K$, $f\in L[V]$ and $c\models p|L$.

For two generically stable types $p$ and $q$ we denote by $p(x)\otimes q(y)$ the generically stable type $\tp(a,b/\mathbb{U})$ where $b\models q|\mathbb{U}$ and $a\models p|\mathbb{U}b$. It is equal to $q(y)\otimes p(x)$. We shorthand $p^{\otimes n}$ for $p\otimes\dots\otimes p$ ($n$ times).

\begin{example}\label{E:p_O-minimal-val}
Let $p_\mathcal{O}$ be the global generic type of the closed ball $\mathcal{O}$ (in the sense of \cite[Definition 7.17]{StableDomination}), i.e  $p_\mathcal{O}(x)$ says that $x\in\mathcal{O}$  but $x$ is not in any proper sub-ball of $\mathcal{O}$. One sees that for every polynomial in $n$ variables $f\in K[X]$ and $c\models p_\mathcal{O}^{\otimes n}|K$, $\val (f(c))=\min_i\{\val (b_i)\}$, where $\{b_1,\dots,b_m\}$ are the coefficients of $f$.
\end{example}

If $p$ and $q$ are generically stable types concentrated on some (type-)definable group $G$ then we write $pq$ for $f_*(p\otimes q)$, where $f(x,y)=x\cdot y$, and likewise $p^n$ for $h_*p^{\otimes n}$ where $h(x_1,\dots,x_n)=x_1\cdot\ldots\cdot x_n$.

\subsection*{Generically stable groups}
We review some definitions from \cite{Metastable}.

Let $G$ be a type-definable group, $p$ an $A$-definable type on $G$ and $g\in G$. The \emph{left translate} of $p$ by $g$, $gp$,  is the definable type such that for any $A\cup {g}\subseteq A^\prime$, \[d\models p|A^\prime \Leftrightarrow gd\models gp|A^\prime.\] Similarly the right translate $pg$.

\begin{definition}
Let $G$ be a type-definable group. A definable type $p$ concentrated on $G$ is \emph{left generic} if for any $A=\acl(A)$ over which it is defined and $g\in G$, $pg$ is definable over $A$. Similarly right generic.
\end{definition}

Any generically stable left generic is also right generic.

\begin{definition}\label{D:gen-stab-grp}
A type-definable group will be called \emph{generically stable} if it has a generically stable generic.
\end{definition}

For a generically stable group $G$, let $G^0$ be the intersection of all definable subgroups of finite index. It is of bounded index and called the \emph{connected component} of $G$. We have the following: $p$ is the unique generic type of $G$ if and only if for all $g\in G$, $gp=p$ if and only if $G=G^0$. If $G=G^0$ then we say that $G$ is \emph{connected}.

\subsection*{Algebraic Geometry}
We will assume basic knowledge of schemes and group schemes, see \cite{gortz,stacks-project} for the relevant definitions. 

For a field $F$, by a variety over $F$ we mean a geometrically integral separated scheme of finite type over $F$. We will reserve  capital letters $V,U,W,\dots$ for varieties and calligraphic letters $\cV,\cU,\cW,\dots$ for general schemes, usually over valuation rings.

By an algebraic group over $F$, we mean a geometrically integral group scheme $G$ of finite type over $F$.

For an affine scheme $\cV=\Spec A$ we denote by $D_\cV(f)$  the \emph{principal open affine subscheme} defined by $f$, i.e $\Spec A_f$ for some $f\in A$. For a projective variety we denote by $D_{+,\cV}(f)$ the projective analog. Similarly for affine or projective varieties.

For any variety $V$ over $F$ and  field $L\supseteq F$ we denote by $V_L:=V\times_{\Spec F} \Spec L$ the base change of $V$ to $L$. Since $V$ is geometrically integral over $F$, $V_L$ is still a variety over $L$. 

Let $F$ be a valued field. Most of our schemes over $\cO_F$ will be quasi-compact integral and faithfully flat separated dominant over $\cO_F$. For a valued field extension $L\supseteq F$, the base change $\cV_{\cO_L}=\cV\times_{\Spec \cO_F}\Spec \cO_L$ is still quasi-compact and faithfully flat separated  dominant over $\cO_L$. If we further assume that $\cV_F$ is a variety over $F$ then by Fact \ref{F:integrality goes using flatness}, $\cV_{\cO_L}$ is also an integral scheme.  As $\cV$ is separated over $\cO_F$, by the valuative criterion for separatedness \cite[01KZ]{stacks-project}, we have $\cV_{\cO_L}(\cO_L)\subseteq \cV_L(L)$ for any valued field extension $L\supseteq F$.

We recall that by \cite{nagata}, every flat scheme of finite type over $\cO_F$ is finitely presented over $\cO_F$; we will use this fact freely.

We denote by $V$ the definable set $V_\mathbb{K}(\mathbb{K})$, and by $\cV(\cO)$ the pro-definable set $\cV_{\cO_{\mathbb{K}}}(\cO)$, hopefully without causing too much confusion. 

Given such a scheme $\cV$ with $\cV_F$ a variety, we can view $\cV(\cO)$ as a type-definable subset of the definable set $\cV_F$ or as a pro-definable subset of an inverse limit of powers of $\cO$, in (possibly) infinitely many variables. These two sets are obviously in pro-definable bijection. There are, however, different advantages to each of these presentations:

It will be convenient to view it as a type-definable subset of $\cV_F$ when, e.g., $\cV$ is a group scheme over $\cO_F$, and then $\cV(\cO)$ is a type-definable subgroup of $\cV_F$. It will be convenient  to view it as a pro-definable set when, e.g., we consider the pro-definable  reduction map $r:\cV(\cO)\to \cV_{\bk_F}$, where $\cV_{\bk_F}=\cV\times_{\cO_F}\bk_F$ is the special fiber.  In this situation the reduction map is just given by $\res$. If we view it has a type-definable subset of $\cV_F$ then, in general, the reduction map $r$ is not equal to the $\res$ map on $\cV(\cO)$.

We refer to \cite[Section 2.2]{stab-pointed} for an explanation how these can be thought of as (pro-)definable sets.

\subsection{The Functor}
We review some necessary definitions and results from \cite{stab-pointed}. The results of this section will be tacitly assumed throughout the paper without further reference.
%
%

Let $K$ be an algebraically closed valued field. The category $\GVar$, of \emph{stably pointed varieties}, is the category of pairs $(V,p)$, where $V$ is a variety over $K$ and $p$ is a Zariski dense generically stable type over $K$, concentrated on $V$. Maps are morphisms between varieties that map the types accordingly. The category $\aGVar$ is the restriction of $\GVar$ to affine varieties.

We define a functor from $\aGVar$ to the category of affine schemes over $\cO_K$ given by $\Phi(V,p)=\Spec K[V]^p$, where \[K[V]^p=\{f\in K[V]: p\vdash \val(f(x))\geq 0\}.\] The scheme $\cV:=\Phi(V,p)$ is an affine integral scheme flat over $\cO_K$ \cite[Proposition 4.2.4]{stab-pointed}.  Note that by definition $p$ is concentrated on $\cV(\cO)$. Moreover, $V\cong \cV_K:=\cV\times_{\cO_K}K$ and $\cV$ enjoys the maximum modulus principle with respect to $p$ (mmp w.r.t. $p$): 

For every regular function $f$ on $\cV_{\mathbb{K}}$  there is some $\gamma_f\in \Gamma$ such that $p\vdash \val(f(x))=\gamma_f$ and for any $h\in \cV(\cO)$, $\val(f(h))\geq \gamma_f$.
As a result if $\cU\subseteq \cV$ is an open affine subscheme with $\cU(\cO)\neq \emptyset$ then $p$ is concentrated on $\cU(\cO)$. For the definition of the maximum modulus principle for general quasi-compact separated schemes over $\cO_K$ see \cite[Definition 4.3.1]{stab-pointed}.

If either $K$ is sufficiently saturated or $\cV$ is of finite type over $\cO_K$ then $\cV$ has an $\cO_K$-point. Thus $\Phi(V_\mathbb{K},p)$ is faithfully flat and dominant over $\cO$ \cite[Proposition 3.1.4]{stab-pointed}. On the other hand, by \cite[Proposition 4.2.15]{stab-pointed} $\Phi(V,p)\times_{\cO_K}\cO=\Phi(V_\mathbb{K},p)$ and thus by faithfully flat descent  $\cV=\Phi(V,p)$ is faithfully flat over $\cO_K$ as well \cite[Corollary 14.12]{gortz} (though might not have an $\cO_K$-point) and dominant over $\cO_K$ by \cite[Corollaire IV.2.6.4]{EGA}.

%
%

%
%

We now expand the above to a larger family of valued fields. 

\begin{definition}
\begin{enumerate}
\item We say that a non-trivially valued field $(F,\val)$ is \emph{gracious} if it is perfect henselian defectless with  perfect residue field. 

\item  We say that a generically stable $F$-definable type  $p$ is \emph{strictly based on $F$} if $\Gamma_{F(c)}=\Gamma_F$ for $c\models p|F$.
\end{enumerate}
\end{definition}

\begin{remark}
Note that if $F$ has a divisible value group then every generically stable type over $F$ is strictly based on $F$.
\end{remark}

Let $\aGVar[F]$ be the category of pairs $(V,p)$, where $V$ is an affine variety over $F$ and $p$ is a generically stable $F$-definable type, strictly based on $F$ and Zariski dense in $V$. Likewise, $\GVar[F]$ for general varieties.

For a gracious valued field $F$ and $(V,p)\in\aGVar[F]$ we let
\[F[V]^p=\{f\in F[V]:p\vdash \val(f(x))\geq 0\},\]
and $\Phi_F(V,p)=\Spec F[V]^p$. So $\Spec \Phi_F(V,p)$ is an integral affine scheme. The following is an important result allowing us to apply descent arguments.

%
%
%

\begin{fact}\cite[Corollary 4.2.13, Proposition 4.2.15]{stab-pointed}
Let $F$ be a gracious field, $K\supseteq F$ a model of ACVF and $(V,p)\in \aGVar[F]$. Then
\[\Phi_{K}(V_{K},p)=\Phi_F(V,p)\times_{\Spec \cO_F}\Spec \cO_{K}.\]
\end{fact}

As a result, by faithfully flat descent, $\Phi_F(V,p)$ is faithfully flat over $\cO_F$ \cite[Corollary 14.12]{gortz} and dominant over $\cO_F$ by  \cite[Corollaire IV.2.6.4]{EGA}. Due to the above fact we usually omit the subscript from $\Phi_F$, unless a precise distinction is needed. In particular, the definable set $\Phi(V,p)(\cO)$ is well-defined.

If $\Phi(V,p)$ is of finite type over $\cO_F$ then by flatness it is finitely presented over $\cO_F$ \cite{nagata}.

The functor $\Phi$ commutes with products in the following sense: If $(V,p),(U,q)\in \aGVar[F]$ and $p\otimes q$ is strictly based on $F$ then $\Phi(V\times_F U,p\otimes q)=\Phi(V,p)\times_{\Spec \cO_F}\Phi(U,q)$, see \cite[Proposition 4.2.18]{stab-pointed}.

We do not know if $\Phi$ commutes with open immersions, e.g. for $U\subseteq V$ an open subvariety, we do not know if $\Phi(U,p)$ is an open subscheme of $\Phi(V,p)$. However, by passing to an open subvariety of $U$ we can always assume that this is the case (taking $W=U$ below):

\begin{lemma}\label{L:locally preserves open immersions}
Let $F$ be a gracious valued field. Let $(V,p),(U,q)\in \aGVar[F]$, $W\subseteq U$ an open affine subvariety and $\varphi: W\to V$ an open immersion with $\varphi_*q=p$. Then there exist $f\in F[U]$ and $g\in F[V]$, with $q\vdash \val(f(x))=0$ and $p\vdash \val(g(x))=0$, such that 
\begin{enumerate}
\item $D_U(f)\subseteq W$ and $\varphi$ restricts to an isomorphism $\varphi'$ between $D_U(f)$ and $D_V(g)$
\item For $\cV=\Phi(V,p)$ and $\cU=\Phi(U,q)$, $\Phi(D_U(f),q)=D_{\cU}(f)$ and $\Phi(D_V(g),p)=D_\cV(g)$ 
\item $\Phi(\varphi'):D_{\cU}(f)\to D_{\cV}(g)$ is an isomorphism. 
\end{enumerate}
\end{lemma}
\begin{proof}
By a standard argument, there are $f\in F[U]$ and $g\in F[V]$ such that $D_U(f)\subseteq W$, $D_V(g)\subseteq f(W)$ and that $\varphi$ restricts to an isomorphism between $D_U(f)$ and $D_V(g)$. This gives (1). As $p$ and $q$ are strictly based on $F$ we may further assume that $q\vdash \val(f(x))=0$ and $p\vdash \val(g(x))=0$ (since $\val(f(c))\in \Gamma_F$ for $c\models q|F$, and likewise for $g$ and $p$). In particular, $(F[U]_f)^q=(F[U]^q)_f$ and likewise for $g$. Thus if we set $\cU=\Phi(U,q)$ and $\cV=\Phi(V,p)$ we get (2).  

Item (3) follows since $\Phi$ obviously preserves isomorphisms.
\end{proof}
 
\section{Some more on the functor}
For the entirety of this section  we let $F$ be a gracious valued field and $K$  an algebraically closed valued field containing $F$. The aim of this section is to prove some more properties of $\Phi$.

\subsection{Integrality of the Fibers}
 
In this subsection we prove that given $(V,p)\in\aGVar[F]$, the fibers of $\Phi(V,p)\to \Spec\cO_F$ are geometrically integral. We then deduce, assuming $\Phi(V,p)$ is of finite type over $\cO_F$, that $\Phi(V,p)$ is generically smooth over $\cO_F$.

We denote by $(K,\val)^{sh}$ the Shelah expansion of $(K,\val)$, see \cite[Section 3.3]{guidetonip} and Appendix \ref{S:generically stable and shelah}. We will require the following result whose proof we postpone to the appendix, see Proposition \ref{P:gen-stable-sh}.

\begin{proposition}\label{P:extension to sh}
Let $T$ be a complete NIP theory in a first order language $\mathcal{L}$ and $M\models T$. Let $\mathbb{U}\succ M$ be a large saturated model of $T$ and $\mathbb{U}^*$ an expansion to a saturated model of $\mathrm{Th}(M^{sh})$ (so $\mathbb{U}^*\restriction \mathcal{L}=\mathbb{U}$). 

For every global type $p$, if $p$ is generically stable over $M$ then there exists a unique extension $p\subseteq q\in S^{\mathcal{L}^{sh}}(\mathbb{U}^*)$ which is generically stable over $M$.
\end{proposition}

The following observation is simple, yet useful:

\begin{fact}\label{F:tensor-product-same}
Let $R\subseteq S$ be (possibly trivial) valuation rings with fraction field $F$. Then for every $S$-algebras $A$ and $B$
\[A\otimes_R B=A\otimes_S B.\]
\end{fact}
\begin{proof}
Let $a\in A$, $b\in B$ and let $r\in S\setminus R$, so $r^{-1}\in R$.
\[(ra)\otimes b=(ra)\otimes (r^{-1}rb)=(rr^{-1})a\otimes (rb)=a\otimes(rb),\] as needed. 
\end{proof} 

\begin{proposition}\label{P:each fiber is integral}
Let $(V,p)\in\aGVar[F]$ and let $\cV=\Phi(V,p)$. All the fibers of $\cV\to \Spec\cO_F$ are geometrically integral.
\end{proposition}
\begin{proof}
After base change, we may assume that $F=K$.
 
Let $v$ be the valuation on $K$, $\mathcal{O}_v:=\mathcal{O}_K$ and $\Gamma_v:=\Gamma_K$. Let $\mathfrak{p}_w \in \Spec \mathcal{O}_K$ be a prime ideal corresponding to a coarsening $w$ of $v$ with value group $\Gamma_w:=\Gamma_v/\Delta_w$ and residue field $\bk_w$; so $\mathcal{O}_w:=(\mathcal{O}_v)_{\mathfrak{p}_w}$ is the corresponding valuation ring.

If $\mathfrak{p}_w$ corresponds to the generic point, the generic fiber is just $\cV\times_{\Spec \cO_v}\Spec K=V$ so integral. Otherwise, $w$ is a non-trivial coarsening of $v$.

Let $(K,\val)^{sh}$ be the Shelah expansion of $(K,\val)$. Let $\tilde p$ be the (unique) generically stable extension of $p$ living in the saturated expansion $(\mathbb{K},\val )^*$, as given by Proposition \ref{P:extension to sh}. Possibly after passing to $T^{eq}$, which we can do by Fact \ref{F:fact-genstable}, the valuation $w$ is now definable in $(K,v)^{sh}$ and thus defines a coarsening of $v$ in $\mathbb{K}$ as well, which we will also denote by $w$. Let $p_w$ be the restriction of $\tilde p$ to $(\mathbb{K},w)$;  $p$ is just the restriction of $\tilde p$ to $(\mathbb{K},v)$. Note that $p_w$ is still $K$-definable and Zariski dense in $V$. It is generically stable by Fact \ref{F:fact-genstable}.
Let $K[V]^{p_w}=\{f\in K[V]:w(f(c))\geq 0,\text{ for } c\models p_w|K\}$.
\begin{claim}
Viewing $K[V]^{p_v}\otimes_{\cO_v}\cO_w$ as a subset of $K[V]^{p_w}$,
\[K[V]^{p_w}=K[V]^{p_v}\otimes_{\cO_v}\cO_w.\] 
\end{claim}
\begin{claimproof}
Let $c\models \tilde p|K$, then obviously $c\models p_v|K$ and $c\models p_w|K$.

Let $f\in K[V]$ be a regular function satisfying $w(f(c))\geq 0$. As $p_w$ is generically stable, $w(f(c))\in \Gamma_w$. Since $w$ is a coarsening of $v$, there are two possibilities for $v(f(c))$, either $v(f(c))\geq 0$ or $0>v(f(c))\in \Delta_w$. If $v(f(c))\geq 0$ there is nothing show. If $0>v(f(c))\in \Delta_w$ then let $a\in K$ be such that $v(f(c))=v(a)$. Such an element exists seen $p_v$ is orthogonal to $\Gamma_v$. Writing $f(c)=a\cdot \frac{f(c)}{a}$ concludes the proof since $v\left( \frac{f(c)}{a}\right)=0$ and $v(a)\in \Delta_w$ and hence $w(a)=0$.
\end{claimproof}

The fiber over $\mathfrak{p}_w$ is by definition $\cV\times_{\Spec \cO_v}\Spec \bk_w$, or in other words \[\Spec\left( K[V]^p\otimes_{\cO_v}\bk_w\right).\]
As $\bk_w=\cO_w\otimes_{\cO_w}\bk_w$, by Fact \ref{F:tensor-product-same} this is equal to 
\[\Spec\left(K[V]^{p_v}\otimes_{\cO_v}\cO_w\otimes_{\cO_v}\bk_w\right).\]
By the claim above and Fact \ref{F:tensor-product-same} again we get that it is equal to
\[\Spec\left(K[V]^{p_w}\otimes_{\cO_w}\bk_w\right).\]

It is thus sufficient to prove the following claim:
\begin{claim}
For any $(V,p)\in\aGVar$, $\Phi(V,p)$ has an integral special fiber.
\end{claim}
\begin{claimproof}
We show that the special fiber $K[V]^p\otimes_{\cO_K}\bk_K\cong K[V]^p/(\m_K K[V]^p)$ is an integral domain. Let $f,g\in K[V]^p$ and assume that $fg\in \m_KK[V]^p$. In particular $\val(f(c)g(c))>0$ for $c\models p|K$ and thus without loss of generality $\val(f(c))>0$. Since $p$ is generically stable, there is some $a\in K$ such that $\val(f(c))=\val(a)>0$ so $\val(f(c)/a)=0$ and hence $f(c)/a\in K[V]^p$. Consequently, $f=a\cdot f/a\in \m_K K[V]^p$, as needed.
\end{claimproof}
Apply the last claim to $(K,w)$ and $p_w$ and conclude.
\end{proof}

%

We now turn to showing smoothness.

By \cite[Theorem 3]{nagata}, a flat scheme of finite type over a valuation ring is finitely presented. Actually the proof gives a bit more (see \cite[page 159]{nagata}):

\begin{fact}\label{F:refined nagata}
Let $R$ be a valuation ring with maximal ideal $\m$ and let $R[X]$ be the polynomial ring over $R$ with variables $X=(x_1,\dots,x_n)$. Then for every $I\trianglelefteq R[X]$ such that  $R[X]/I$ is flat over $R$, if $f_1+\m[X],\dots, f_r+\m[X]$ generate $(I+\m[X])/\m[X]$ then $f_1,\dots, f_r$ generate $I$.
\end{fact}

%
%
%
%

\begin{proposition}\label{P:non empty smooth locus}
Let $\cV$ be an integral affine scheme of finite type over $\cO_K$ with an $\cO_K$-point. If the special fiber $\cV_{\bk_K}$ is reduced then the smooth locus of $\cV$ over $\cO_K$ is open non-empty and has an $\cO_K$-point. 
\end{proposition}
\begin{proof}
Assume that $\cV=\Spec \cO_K[X]/I$. Since $\cV$ has an $\cO_K$-point it is faithfully flat over $\cO_K$ by \cite[Proposition 3.1.4]{stab-pointed}, finitely presented over $\cO_K$ by \cite[Theorem 3]{nagata} and $\cV_{\bk_K}(\bk_K)$ is non empty.

%

The coordinate ring of the special fiber is isomorphic to
\[(\cO_K[X]/I)\otimes_{\cO_K}\bk_K\cong \cO_K[X]/(I+\m_K[X])\cong \left(\nicefrac{\cO_K[X]}{\m_K[X]}\right)/\left(\nicefrac{I+\m_K[X]}{\m_K[x]}\right);\] by Fact \ref{F:refined nagata} there are generators $f_1,\dots,f_r$ of $I$ such that $\overline{f_1},\dots,\overline{f_r}$, with $\overline{f_i}=f_i+\m_K[X]$, are non-zero and generate $\nicefrac{I+\m_K[X]}{\m_K[x]}$.

As $\cV_{\bk_K}$ is reduced (and of finite type over $\bk_K$) and $\bk_K$ is perfect, by \cite[Theorem 6.19]{gortz} the singular locus  of $\cV_{\bk_K}$ is a proper closed subvariety. Thus there is some $a\in \cV_{\bk_K}(\bk_K)$ and a minor $\det \overline{M}(a)$ of $\left( \frac{\partial \overline{f_i}}{\partial x_j}(a)\right)_{1\leq i\leq r,\, 1\leq i\leq n}$ of order $n-d$  with $d=\dim \cV_{\bk_K}$ and $\det \overline{M}(a)\neq 0$.

By \cite[Theorem 3.2.4]{stab-pointed}, there is some $b\in \cV(\cO_K)$ with $r(b)=a$, where $r:\cV(\cO_K)\to \cV_{\bk_K}(\bk_K)$ is the reduction map.  Viewing $b$ as an element of $\cO_K^n$ we exactly get that $\res(b)=a$. Hence $\val(\det M(b))=0$ where $\det M(b)$ is the corresponding minor of  $\left( \frac{\partial f_i}{\partial x_j}(b)\right)_{1\leq i\leq r,\, 1\leq i\leq n}$. This exactly means that the prime ideal $\mathfrak{p}_b=\{f\in \cO_K[X]/I: \val(f(b))>0\}$ of $\cO_K[X]/I$ is a smooth point of $\cV$ over $\cO_K$ of relative dimension $d$. 

Thus the smooth locus is open non-empty and contains the $\cO_K$-point $b$.
\end{proof}  

\begin{corollary}\label{C::can reduce to smooth}
Let $(V,p)\in\aGVar[F]$. If $\Phi_F(V,p)$ is of finite type over $\cO_F$ then there exists an open affine subvariety $U\subseteq V$ such that $\Phi_F(U,p)$ is an open affine subscheme of $\Phi_F(V,p)$ and $\Phi_F(U,p)$ is smooth over $\cO_F$.
\end{corollary}
\begin{proof}
%
The $\cO_K$-scheme $\cV_{\cO_K}=\Phi_K(V_K,p)$ satisfies the assumptions of Proposition \ref{P:non empty smooth locus} (by Proposition \ref{P:each fiber is integral}). As $\cV$ is affine, the smooth locus of $\cV$ over $\cO_F$ is a finite union of principal open affine subschemes of $\cV$ (given by Jacobian determinants). Base changing this open subscheme to $\Spec \cO_K$ we get the smooth locus of $\cV_{\cO_K}$.

Hence there exists some $f\in F[V]^p$ such that $D_{\cV_{\cO_K}}(f)$ is smooth over $\cO_K$ and $D_{\cV_{\cO_K}}(f)(\cO_K)\neq \emptyset$. As smoothness descends under faithfully flat descent \cite[IV.17.7.3]{EGA}, $\cU:=D_\cV(f)$ is smooth over $\cO_F$.

Since $\cU(\cO)$ is non empty, $p$ is concentrated on $\cU(\cO)$, i.e.  $\val(f(c))=0$ for $c\models p|F$. Consequently,  $(F[V]^p)_{f}=(F[V]_{f})^p$ and thus, setting $U=\Spec F[V]_f$, $\Phi_F(U,p)=\cU$.
\end{proof}

\subsection{Strongly Stably dominated types}

We recall the definition of a strongly stably dominated type (at least an equivalent definition, see \cite[Proposition 8.1.2]{Non-Arch-Tame}).

\begin{definition}\label{D: strongly stab dominated}
Let $L$ be a valued field and $p$ an $L$-definable type on a variety $V$ over $L$. We say that $p$ is \emph{strongly stably dominated} if there exists some $L$-definable map $g$ into a variety over the residue field for which $\dim(p)=\dim(g_*p)$.
\end{definition}

Every strongly stably dominated type is generically stable; not every generically stable type is strongly stably dominated, but every generically stable type concentrated on an algebraic curve is strongly stably dominated \cite[Lemma 8.1.4	and Remark 8.1.5]{Non-Arch-Tame}.

\begin{fact}\label{F:strongly stably implies existence of finite morphism}
Let $L$ be a valued field, $V$ an affine variety over $L$, with $n=\dim(V)$, and $p$ a Zariski dense $L$-definable type concentrated on $V$. If $p$ is strongly stably dominated then there exist open affine subvarieties $V'\subseteq V$, $ U\subseteq \mathbb{A}^n_L$ and a finite morphism $f:V'\to U$ with $f_*p=p_\cO^{\otimes n}$.
\end{fact}
\begin{proof}
By \cite[Lemma 8.1.2]{Non-Arch-Tame}, we can find such $f:V_0\to \mathbb{A}^n_L$, $V_0\subseteq V$ open affine subvariety, with $f$ quasi-finite. By \cite[Tag 02NW]{stacks-project}, we may further reduce (the target and domain of) $f$ to achieve finiteness.
\end{proof}

The purpose of this section is to present some results on strongly stably dominated types. We start with some known results.

\begin{fact}\cite[Lemmas 2.13, 2.14]{HiHrSi}\label{F:strongly stably dominated groups}
\begin{enumerate}
\item The generic types of a \emph{definable} generically stable group are strongly stably dominated.
\item If the unique generically stable type of a type-definable connected generically stable subgroup $H$ of an algebraic group $G$ is strongly stably dominated then $H$ is definable.
\end{enumerate}
\end{fact}

For the next result we will need the following, which is hidden inside the proof of \cite[Theorem 3.2.4]{stab-pointed}. 

\begin{fact}\label{F: empty O-points implies empty k-points}
Let $\cV$ be an affine integral scheme flat and dominant over $\cO_K$. If $\cV(\cO)=\emptyset$ then $\cV_{\bk_K}=\emptyset$. If $\cV$ is of finite type over $\cO_K$ then $\cV(\cO_K)=\emptyset$ implies that $\cV_{\bk_K}(\bk_K)=\emptyset$.
\end{fact}
\begin{proof}
Let $\cV=\Spec \cO_K[X]/I$ be the given affine scheme. Assume that $\cV_{\bk_K}\neq \emptyset$, after base-changing we also assume that $\cV_{\bk_K}(\bk_K)\neq \emptyset$. By \cite[Lemma 3.2.2(2)]{stab-pointed}, there exists $b\in \cV(\cO)$. If $\cV$ is of finite type over $\cO_K$ (so finitely presented over $\cO_K$) then $\cV(\cO_K)\neq \emptyset$ as well.
\end{proof}

We observe the following:

\begin{fact}\label{F:integrality goes using flatness}
Let $\mathcal{V}$ be a scheme which is flat over $\cO_F$ and assume that $\cV_F$ is geometrically integral, then $\cV$ and $\mathcal{V}\times_{\Spec \cO_F}\Spec \cO_K$ are integral schemes.
\end{fact}
\begin{proof}
We may assume that $\mathcal{V}=\Spec A$ is affine. 
Since $A$ is flat over $\cO_F$, 
\[A\otimes_{\cO_F}\cO_K\hookrightarrow A\otimes_{\cO_F}K.\]
We also have $A\otimes_{\cO_F}K\cong A\otimes_{\cO_F}(F\otimes_F K)\cong (A\otimes_{\cO_F}F)\otimes_F K$, and since the latter is an integral domain (by assumption), $A\otimes_{\cO_F}\cO_K$ is an integral domain. 

Likewise, since $A$ is flat over $\cO_F$, 
\[A\otimes_{\cO_F}\cO_F\hookrightarrow A\otimes_{\cO_F}\cO_K,\] so $A$ is an integral domain as well.
%
%
%
\end{proof}

In the following, it is convenient to view $\cV(\cO)$ as a pro-definable subset of an inverse limit of (powers of) $\cO$. Thus $b$ might in general be an infinite tuple.

\begin{lemma}\label{L:global sections residual mmp}
Let $\cV$ be an affine integral scheme flat over $\cO_K$, with $\cV_K$ geometrically integral, and assume that it has the mmp w.r.t. $p$, a generically stable type over $K$ concentrated on $\cV(\cO)$. If $\cV_{\bk_K}$ is integral then for any $b\models p|K$, $\cV_{\bk_K}=\Spec(\bk_K[\res(b)])$. In particular, $r_*p$ is the unique generic type of $\cV_\bk$.
\end{lemma}
\begin{proof}
Assume that $\cV=\Spec \cO_K[X]/I$. Since $p$ is concentrated on $\cV(\cO)$, $\cV$ has  an $\cO$-point so flat over $\cO$. Thus, as $\cV_\cO$ is integral by Fact \ref{F:integrality goes using flatness},  $\cV_\cO$ is dominant and faithfully flat over $\cO$ \cite[Proposition 3.1.4 and its proof]{stab-pointed}; by faithfully flat descent so is $\cV$ over $\cO_K$. 

Let $b\models p|K$. By \cite[Lemma 3.2.2(1)]{stab-pointed}, $\cO_K[X]/I\cong \cO_K[b]$, given by $X\mapsto b$. Now let $\cO_K[b]\otimes_{\cO_K}\bk_K\to \bk_K[\res(b)]$ be the surjection given by the residue map; we claim it is injective. Note that every element of $\cO_K[b]\otimes_{\cO_K}\bk_K$ is of the form $f(b)\otimes 1$.

 Let $f\in \cO_K[X]/I$ and assume that $\overline{f}(\res(b))=0$, i.e. $\val(f(b))>0$. Thus $\val(f(c))>0$ for any $c\models p|L$ and any large enough model $L$ containing $K$. By the maximum modulus principle, $D_\cV(f)(\cO)=\emptyset$, where $D_\cV(f)$ is the principal open subscheme of $\cV$ given by $f$. 

Note that $D_\cV(f)$ is also integral, flat and dominant over $\cO_K$ as an open subscheme of such; so by Fact \ref{F: empty O-points implies empty k-points} we get that $D_\cV(f)_\bk(\bk)=\emptyset$ which gives $f=0$ in $\cO_K[b]\otimes_{\cO_K}\bk_K\subseteq \cO_K[b]\otimes_{\cO_K}\bk$.
\end{proof}

\begin{lemma}\label{L: strongly dominated using dimensions}
Let $(V,p)\in\aGVar$. The type $p$ is strongly stably dominated if and only if $\dim V=\dim \cV_{\bk_K}$, for $\cV=\Phi(V,p)$. Furthermore, in this situation $p$ is stably dominated via $r:\cV(\cO)\to \cV_{\bk_K}$.
\end{lemma}
\begin{proof}
Assume that $\dim V=\dim \cV_{\bk_K}$. Since $\cV\to \Spec\cO_K$ has integral fibers (Proposition \ref{P:each fiber is integral}), by Lemma \ref{L:global sections residual mmp} $r_*p$ is the generic type of $\cV_{\bk_K}$, so $\dim(p)=\dim V=\dim \cV_{\bk_K}=\dim(r_*p)$, so $p$ is strongly stably dominated. By \cite[Proposition 4.1.4]{stab-pointed}, $p$ is stably dominated via $r$.

For the other direction, by Fact \ref{F:strongly stably implies existence of finite morphism} there exist  open affine subvarieties $U\subseteq V$, $W\subseteq \mathbb{A}^n_K$ and a finite morphism $f:U\to W$ satisfying $f_*p=p_\cO^{\otimes n}$ for $n=\dim V$. Writing $f=(f_i)$, then  $f_i\in K[U]^p$ for all $i$.  Letting $\overline{f_i}=\res\circ f_i$, by Lemma \ref{L:global sections residual mmp} we have that for $c\models p|K$ \[\overline{f_i}(c)=\res(f_i(c))\in \bk_K[\res(c)]\cong K[U]^p\otimes_{\cO_K}\bk_K.\] 
Since $f(c)\models p_\cO^{\otimes n}$ this implies that $\dim \bk_K[\res(c)]=n$, i.e. $\dim \cV_{\bk_K}=n$.
\end{proof}

The following can be seen as a ``strengthening'' of \cite[Lemma 5.1]{Metastable}.

\begin{lemma}\label{L:lemma 5.1-expanded}
Let $A$ be a commutative algebraic group over $K$ and $p$ be a strongly stably dominated type concentrated on $A$, which is definable over $K$. If $p^n$ is Zariski dense in $A$ then it is the generic of a coset of a generically stable connected definable subgroup of $A$, where $n=\dim A$.
\end{lemma}
\begin{proof}
We may assume that $n>0$.
By \cite[Lemma 5.1]{Metastable}, there exists a type-definable generically stable connected subgroup $H$ of $A$ with unique generic $p^{\mp 2n}$, where \[p^{\mp 2n}=f_*p^{\otimes 2n}\] for $f:A^{2n}\to A$ sending $(a_1,\dots,a_{2n})\mapsto a_1^{-1}\cdot a_2\cdot a_3^{-1}\cdot \ldots \cdot a_{2n}$ and, moreover, $p$ is concentrated on a coset of $H$. Let $a\in A$ be such that $ap$ is concentrated on $H$, where $ap$ is the pushforward of $p$ under the multiplication-by-$a$ map.

Since $p^{\mp 2n}$ is Zariski dense (because $p^n$ is Zariski dense) and strongly stably dominated, by Fact \ref{F:strongly stably dominated groups} $H$ is definable. By \cite[Proposition 4.6]{Metastable}, there exists an $\omega$-stable connected  group $\mathfrak{h}$ in the residue field sort and a definable surjection $\phi:H\to \mathfrak{h}$, such that the generic of $H$ is stably dominated by the generic of $\mathfrak{h}$ via $\phi$, and
 $\dim H=\dim(p^{\mp 2n})=\dim(\phi_*p^{\mp 2n})=\dim \mathfrak{h}=n$. Since $H$ is Abelian, $\mathfrak{h}$ is necessarily Abelian as well.

\begin{claim}
$\phi_*((ap)^n)=\phi_*p^{\mp 2n}$.
\end{claim}
\begin{claimproof}
Let $q:=\phi_*(ap)$. Since $\dim(\mathfrak{h})=n$, $\dim(q^n\cdot q^n)=\dim(q^n)$ (e.g., since $\dim$ is equal to the Morley rank).  By \cite[Lemma 1.3]{kowalski}, $q^n$ is a generic of a coset of a type-definable subgroup of $\mathfrak{h}$ with generic $q^nq^{-n}=q^{\mp 2n}$. Since $\mathfrak{h}$ is connected with unique generic $q^{\mp 2n}=\phi_*p^{\mp 2n}$, $q^n=\phi_*((ap)^n)=\phi_*p^{\mp 2n}$, as needed.
\end{claimproof}

By \cite[Lemma 4.9]{Metastable}, $(ap)^n$ is a generic of $H$ so, by connectedness of $H$, $(ap)^n=a^np^n=p^{\mp 2n}$, as desired.
\end{proof}

\subsection{A finiteness condition}
The purpose of this section is to provide a condition that guarantees that $\Phi(V,p)$ is of finite type over $\cO_F$, for $(V,p)\in \aGVar[F]$. Then we show that generic types of Abelian schemes over $\cO_K$ satisfy this condition.


\begin{proposition}\label{P:strongly^2 gives finite type}
Let $(V,p)\in \aGVar[F]$.
If there exists a finite morphsim $f:V\to W$, with $W\subseteq  \mathbb{A}^n_F$ an open subvariety, satisfying that $f_*p=p_{\cO}^{\otimes n}$ and for $c\models p|F$, the valuation on $(F(f(c)),\val)$ extends uniquely to $F(c)$, then there exists an open subvariety $U\subseteq V$ such that $\Phi(U,p)$ is of finite type over $\cO_F$.
\end{proposition}
\begin{proof}
Let $c\models p|F$ and let $f:V\to W$ be a finite morphism of varieties over $F$ satisfying the assumptions. By restricting $W$ (and $V$) there is no harm in assuming that $W=D_{\mathbb{A}^n_F}(h)$ is a principal open subvariety of $\mathbb{A}^n_F$; as $p_{\cO}^{\otimes n}$ is generically stable we may further assume that $\val(h(d))=0$ for $d\models p_{\cO}^{\otimes n}|F$.

Assuming $f=(f_1,\dots, f_n)$, we identify $F[D_{\mathbb{A}^n_F}(h)]$ with $F[f_1,\dots, f_n]_h\subseteq F[V]$.

Identifying $F[V]$ with $F[c]$ we have that $f_i(c)\in F[c]$ for $1\leq i\leq n$, and as $f(c)\models p_\cO^{\otimes n}|F$ they constitute a valuation transcendence basis of $F[f(c)]$ over $F$, i.e. for any (multi-)index set $I\subseteq \mathbb{N}$ and $g\in F[f(c)]$, if $g=\sum_{\nu\in I}a_\nu f(c)^\nu$ then
\[\val(g)=\min\{\val(a_\nu)\}.\]

Since $(F,\val)$ is a defectless valued field, by the generalized stability theorem \cite[Theorem 1.1]{defectFVK}, $(F(f(c)),\val)$ is also defectless. So $F(c)/F(f(c))$ is finite defectless. Since, by the assumptions, the valuation on $F(f(c))$ extends uniquely to $F(c)$, by \cite[Lemma 11.15]{kuhlmann}, $(F(c)/F(f(c)),\val)$ has a separating basis (as a valued vector space). I.e., since $p$ is strictly based on $F$, $\Gamma_{F(c)}=\Gamma_F$, there are basis elements $g_1,\dots,g_d\in F(V)$, with $\val(g_i(c))=0$, such that for any $a_1,\dots,a_d\in F(f(c))$,
\[\val(\sum_i a_ig_i(c))=\min_i\{\val(a_i)\}.\]

As $f:V\to W$ is finite, $\mathrm{Frac}(F[f_1,\dots,f_n]_h)\otimes_F F[V]=F(V)$, so there exists $r_0\in F[f_1,\dots,f_n]_h$ for which $g_i\in F[V]_{r_0}$ for all $1\leq i\leq d$; note that $F[V]_{r_0}$ is still finite over $F[f_1,\dots,f_n]_{hr_0}$. We may thus find $r_1\in F[f_1,\dots,f_n]_{hr_0}$ for which $g_1,\dots,g_d$ generate $F[V]_{r_0r_1}$ as a module over $F[f_1,\dots,f_n]_{hr_0r_1}$. Let $r=r_0r_1$ and let $U=D_V(r)$. As before, there is no harm in assuming that $\val(r(c))=0$.

We now show that $1/h,1/r,\, g_i,\, f_j$, for $1\leq i\leq d,\, 1\leq j\leq n$ generate $F[U]^p$ as an $\cO_F$-algebra.

Let $s\in F[U]^p$; there exist $a_1,\dots,a_d\in F[f_1,\dots, f_n]_{hr}$ such that $s=\sum_{i=1}^d a_ig_i$. Write each $a_i$ as $b_i/(hr)^{n_i}$ for $b_i\in F[f_1,\dots, f_n]$. Each $b_i$ can be written as $\sum_{\nu\in I} t_{i,\nu} f^\nu$, with $t_{i,\nu}\in F$.

As $\val(s(c))\geq 0$, by the choice of the $g_i$ we get that $\val(b_i(c))=\val(b_i(c)/(hr)(c)^{n_i})\geq 0$ for $i=1,\dots, d$. So $\val(t_{i,\nu})\geq 0$ for all the $t_{i,\nu}$'s, which gives what we wanted.
%
\end{proof}

\begin{corollary}\label{C: dagger implies defectlesshenselian}
Let $(V,p)\in \aGVar[F]$.
Assume that
\begin{itemize}
\item[$(\dagger)$] there exists a finite morphism $f:V\to W$, with $W\subseteq \mathbb{A}^n_F$ an affine open subvariety, satisfying
\begin{enumerate}
\item[(a)] $f_*p=p_\cO^{\otimes n}$ and
\item[(b)] for any $c\models p|F$, $\tp_{\text{ACF}}(c/F(f(c)))\vdash p|F$.
\end{enumerate}
\end{itemize}

Then there exists an affine open subvariety $U\subseteq V$ such that $\Phi(U,p)$ is of finite type over $\cO_F$.
\end{corollary}
\begin{proof}
Let $c\models p|F$ and let $f:U\to W$ be the finite morphism of varieties over $F$ satisfying $(\dagger)$. 
The fact that the valuation on $(F(f(c)),\val)$ extends uniquely to $F(c)$ follows from condition (b) and quantifier elimination in ACVF.
\end{proof}

We now turn to Abelian schemes over $\cO_K$; our aim is to show that if $\cA$ is an Abelian scheme over $\cO_K$ then $\cA(\cO)=\cA$ is a connected generically stable group and that its generic type  satisfies $(\dagger)$ of Corollary \ref{C: dagger implies defectlesshenselian}.

By an Abelian  scheme over $\cO_K$ we mean a smooth proper group scheme $\mathcal{A}$ over $\cO_K$  whose fibers are geometrically connected. By smoothness, $\mathcal{A}$ is of finite presentation over $\cO_K$ and both $\mathcal{A}_K$ and $\mathcal{A}_{\bk_K}$ are integral (see \cite[Tag 056T]{stacks-project}). For more information on Abelian schemes see \cite[Section 2.3]{abelianschemes}.

 We need the following version of Noether's normalization lemma, but we first recall the following. Let $\mathbb{P}^n_K$ and $\mathbb{P}^n_{\cO_K}$ be the $n$-th projective spaces over $K$ and $\cO_K$, respectively. As $\mathbb{P}^n_{\cO_K}$ is proper over $\cO_K$ (i.e. by clearing denominators), $\mathbb{P}^n_{\cO_K}(\cO_K)=\mathbb{P}^n_{\cO_K}(K)=\mathbb{P}^n_K(K)$. The residue map $\res$ is thus (well-)defined on $\mathbb{P}^n_K(K)$, and the resulting map $r:\mathbb{P}^n_K(K)\to \mathbb{P}^n_{\bk_K}(\bk_K)$ is called the specialization map.

\begin{lemma}\label{L:projective normalization around mmp}
Let $\cV$ be an integral projective scheme of finite type over $\mathcal{O}_K$ with $\cV_{\bk_K}$ integral and $n=\dim \cV_K=\dim \cV_{\bk_K}$. Let $p$ be a generically stable Zariski dense $K$-definable type concentrated on $\cV(\cO)$ and assume that $\cV$ has the mmp w.r.t. $p$. Then there exists a finite morphism $f:\cV_K\to \mathbb{P}^n_K$, mapping $p$ to $p_\cO^{\otimes n}$, which descends to $\bk_K$ in the sense that the following diagram commutes
 \[ \xymatrix{
\mathcal{V}_K(K)\ar[r]^f\ar[d]_{r} & \mathbb{P}^n_K\ar[d]^r(K)\\
\mathcal{V}_{\bk_K}(\bk_K)\ar[r]^{\overline{f}} & \mathbb{P}^n_{\bk_K}(\bk_K)
}
\]
Moreover, the morphism $\overline{f}$ is also finite.
\end{lemma}
\begin{proof}
Assume that $\mathcal{V}_K$ is embedded inside $\mathbb{P}_K^m$ and likewise $\mathcal{V}_{\bk_K}$ inside $\mathbb{P}_{\bk_K}^m$.

By Noether normalization for projective varieties \cite[Corollary 2.29]{mumford}, there exists a linear subspace $\overline{L}\subseteq \mathbb{P}_{\bk_K}^m(\bk_K)$ of codimension $n+1$ disjoint from $\mathcal{V}_{\bk_K}(\bk_K)$. This linear subspace is defined by $n+1$ linearly independent vectors (forms) $\overline{v}_1,\dots, \overline{v_{n+1}}\in \bk_K^{m+1}$. Let $v_1,\dots,v_{n+1}\in \cO_K^{m+1}$ be lifts of these vectors. They are linearly independent over $K$ and thus define a linear subspace $L\subseteq \mathbb{P}_K^m(K)$ of codimension $n+1$. By properness of $\cV$ over $\cO_K$, every $K$-point of $\cV$ is also an $\cO_K$-point, so $L$ is also disjoint from $\cV_K(K)$.

Let $f:\cV_K(K)\to \mathbb{P}_K^n(K)$ be the projection with center $L$ and let $\overline{f}:\cV_{\bk_K}(\bk_K)\to \mathbb{P}_{\bk_K}^n(\bk_K)$ be its reduction; it is the projection with center $\overline{L}$ (by definition). By \cite[Corollary 2.29]{mumford}, they are both finite surjective morphisms.

We are left with verifying that $p$ is mapped to $p_\cO^{\otimes n}$. Let $b\models p|K$. By \cite[Proposition 4.3.5]{stab-pointed}, any affine open subscheme $\cU\subseteq \cV$, with $\cU(\cO)\neq\emptyset$, has the mmp w.r.t. $p$. Thus by Lemma \ref{L:global sections residual mmp}, $\res(b)$ is a $\bk_K$-generic of $\cV_{\bk_K}(\bk_K)$ and since $\overline{f}$ is dominant $\overline{f}(\res(b))\models (p_\bk^{\otimes n})|K$, where $p_\bk$ is the generic type of $\mathbb{P}^1_\bk$. Since the diagram commutes, $f_*p=p_\cO^{\otimes n}$.
\end{proof}

\begin{lemma}\label{L: p is the unique generically stable generic}
Let $\cG$ be an integral separated group scheme of finite type over $\cO_F$, with $\cG_F$ and $\cG_{\bk_F}$ geometrically integral. Then
\begin{enumerate}
\item $\cG(\cO)$ is a connected generically stable group.
\item If $p$ is a $K$-definable type concentrated on $\cG(\cO)$ and there exists an affine open subscheme $\cU\subseteq \cG$ for which $p$ is concentrated on $\cU(\cO)$ and $\cU_{\cO_K}\cong\Phi(\cU_K,p)$ then $p$ is the unique generic type of $\cG(\cO)$.
\end{enumerate}
\end{lemma}
\begin{proof}
By Fact \ref{F:integrality goes using flatness}, we may base change to $K$; so we assume that $F=K$.

$(1)$ By \cite[Proposition 5.1.6]{stab-pointed}, $\cG(\cO)$ is a generically stable group, i.e. has a generically stable $K$-definable generic type $q$. Furthermore, by its proof $r_*q$ is a generic of $\cG_{\bk_K}$ and $q$ is stably dominated by $r_*q$ via $r$. The map $r$ is surjective by \cite[Theorem 3.2.4]{stab-pointed}. For any $g\in \cG(\cO)$, $r_*(gq)=r_*(g)r_*q= r_*q$ by integrality of $\cG_{\bk_K}$ so by \cite[Lemma 4.9]{Metastable} $gq=q$; i.e. $\cG(\cO)$ is connected.

$(2)$. By \cite[Theorem 5.2.2]{stab-pointed}, $\cU$ has the mmp w.r.t. $q$; by the assumption it also has the mmp w.r.t. $p$. Thus for every open affine subscheme $\cW\subseteq \cU$ with $\cW(\cO)\neq \emptyset$,  both types are concentrated on $\cW(\cO)$.  By \cite[Lemma 4.3.7]{stab-pointed}, $p=q$.
%
%
%
\end{proof}

Let $\mathcal{A}$ be an integral Abelian scheme over $\mathcal{O}_K$. By Lemma \ref{L: p is the unique generically stable generic}, $\mathcal{A}(\mathcal{O})=\mathcal{A}_K$ is a connected generically stable group with unique generic type $p$; so by \cite[Theorem 5.2.2]{stab-pointed}, $\cA$ has the mmp w.r.t. $p$. By \cite[Theorem 3.2.4]{stab-pointed}, $r:\cA_K \to \cA_\bk$ is a surjective group homomorphism.

\begin{fact}
Every Abelian scheme over a valuation ring $R$ is projective over $R$.
\end{fact}
\begin{proof}
The result will follow by \cite[Theorem XI.1.4]{raynaud} once we note that the scheme $\Spec R$ is geomtrically unibranch \cite[Tag 0BQ2]{stacks-project} which is straightforward since $R$ is a valuation ring. 
\end{proof}

\begin{proposition}\label{P: generic of Abelian scheme is strongly^2}
Let $\mathcal{A}$ be an integral Abelian scheme over $\cO_K$ and $p$ be its unique generically stable generic type. Then $p$ satisfies $(\dagger)$ of Corollary \ref{C: dagger implies defectlesshenselian}.
\end{proposition}
\begin{proof}
Consider the commutative diagram supplied by Lemma \ref{L:projective normalization around mmp}. Any affine open subscheme $\cU\subseteq \cA$ with $\cU(\cO)\neq \emptyset$ has the mmp w.r.t. $p$, so by Lemma \ref{L:global sections residual mmp}, $r_*p$ is the generic type of $\cA_{\bk_K}$. Let $c\models p|K$ and let $d\in \cA_K=\cA(\cO)$ with $f(c)=f(d)$; so $f(d)\models p_\cO^{\otimes n}$. Since $\overline{f}$ is a finite morphism, necessarily $r(d)\models r_*p|K$.

By \cite[Lemma 4.1.6]{stab-pointed}, $p$ is stably dominated by $r_*p$ via the surjective group homomorphism $r$. By \cite[Lemma 4.9]{Metastable}, and connectedness, we conclude that $d\models p|K$, as needed.
\end{proof}

%
%

\section{Getting a Group Scheme}
Let $F$ be a gracious valued field and $K\supseteq F$ an algebraically closed valued field.

In \cite[Theorem 6.11]{Metastable}, Hrushovksi and Rideau-Kikuchi proved that given an \emph{affine} algebraic group $G$ over $K$ and a generically stable type $p$ concentrating on $G$ satisfying $p^2=p$ there exists an affine group scheme $\mathcal{G}$ over $\cO$ and a pro-definable isomorphism $G\cong \mathcal{G}_K$ under which $\mathrm{Stab}(p)$ is mapped to $\cG(\cO)$. 

We still do not have a full analog for general algebraic groups over $K$. The purpose of this section is to answer this to the affirmative when the functor $\Phi$ (locally) gives a scheme of finite type over $\cO_K$. We then apply this to give a model theoretic criterion for good reduction of (some) Abelian varieties over $K$.

We review some definitions and results from \cite{bosch}. Let $\cS$ be any scheme. Recall the definition of a schematically dense open subscheme \cite[Tags 01RB and 01RE]{stacks-project}; for reduced schemes this definition coincides with Zariski denseness \cite[Tag 056D]{stacks-project}.

\begin{definition}
An open subscheme $\cU$ of a scheme $\cX$ over $\cS$ is called $\cS$-dense if $\cU\times_\cS \cS^\prime$ is schematically dense in $\cX\times_\cS \cS^\prime$ for all morphisms $\cS^\prime\to \cS$.

\end{definition}

\begin{fact}\cite[IV.11.10.10]{EGA}\label{F:S-dense iff schem dense fibers}
If $\cX\to \cS$ is flat and of finite presentation then an open subscheme $\cU$ of $\cX$ is $\cS$-dense if and only if for all $s\in \cS$, the fiber $\cU_s:=\cU\times_\cS \kappa(s)$ is schematically dense in $\cX_s$, where $\kappa(s)$ is the residue class field of $s$ in $\cS$.
\end{fact}
Note that in EGA such denseness is referred to as \emph{universally schematically dense relative to $\cS$} and in SGA3 as \emph{schematically dense in $\cX$ relative to $\cS$}. Furthermore, in \cite[Section 2.5]{bosch} $\cS$-denseness is only defined for smooth $\cS$-schemes and in return they only require that $\cU_s$ be Zariski dense in $X_s$. The reason for this is the following.

\begin{fact}\label{F: meets all fibers so dense}
Assume that $\cX$ is smooth and of finite presentation over $\cS$ and let $\cU\subseteq \cX$ be an open subscheme. Then $\cU$ is $\cS$-dense if and only if $\cU_s$ is Zariski dense in $\cX_s$ for all $s\in \cS$.

In particular, if furthermore for every $s\in \cS$, $\cX_s$ is irreducible and $\cU_s$ is not the empty scheme, then  $\cU$ is $\cS$-dense.
\end{fact}
\begin{proof}
As $\cX\to \cS$ if flat and of finite presentation, if $\cU$ is $\cS$-dense then each $\cU_s$ is schematically dense in $\cX_s$ by Fact \ref{F:S-dense iff schem dense fibers}.

Assume that $\cU_s$ is Zariski dense in $\cX_s$ for all $s\in \cS$. In particular, for any such $s\in \cS$, $\cX_s$ is a smooth scheme over the field $\kappa(s)$ and hence reduced so being Zariski dense is equivalent to being schematically dense by \cite[Tag 056D]{stacks-project}.
\end{proof}
By an \emph{$\cS$-rational map} $\cX\dashrightarrow \cY$ we mean an equivalence class of $\cS$-morphisms $\cU\to \cY$ where $\cU$ is some $\cS$-dense open subscheme of $\cX$. An $\cS$-rational map $\varphi:\cX\dashrightarrow \cY$ is called \emph{$\cS$-birational} if $\varphi$ can be defined by an $\cS$-morphsim $\cU\to \cY$ which induces an isomorphism from $\cU$ onto an $\cS$-dense open subscheme of $\cY$,  see \cite[Section 2.5]{bosch} for more information.

\begin{definition}
Let $\cS$ be a scheme and $\cX$ a smooth separated $\cS$-scheme  finitely presented and  faithfully flat over $\cS$. An $\cS$-birational group law on $\cX$ is an $\cS$-rational map
\[m: \cX\times_\cS \cX\dashrightarrow \cX,\;\;\;\;\;\; (x,y)\mapsto xy,\]
such that
\begin{enumerate}
\item[(a)] the $\cS$-rational maps
\[\Psi_1 : \cX\times_\cS\cX\dashrightarrow \cX\times_\cS \cX,\;\;\;\;\;\; (x,y)\mapsto (x,xy),\]
\[\Psi_2 : \cX\times_\cS\cX\dashrightarrow \cX\times_\cS \cX,\;\;\;\;\;\; (x,y)\mapsto (xy,y),\]
are $\cS$-birational, and
\item[(b)] $m$ is associative; i.e. $(xy)z=x(yz)$ whenever both sides are defined.
\end{enumerate}
\end{definition}

As in Weil's original group chunk theorem an $\cS$-birational group law gives rise to a group scheme:

\begin{fact}\cite[Theorem 6.6.1]{bosch}\label{F:grpchunk}
Let $\cS$ be a scheme, and let $m$ be an $\cS$-birational group law on a smooth and separated $\cS$-scheme $\cX$ which is finitely presented and faithfully flat over $\cS$. Then there exists a smooth and separated $\cS$-group scheme $\overline{\cX}$ of finite presentation with a group law $\overline{m}$ together with an $\cS$-dense open subscheme $\cX^\prime\subseteq \cX$ and an open immersion $\cX^\prime \hookrightarrow \overline{\cX}$ having $\cS$-dense image such that $\overline{m}$ restrict to $m$ on $\cX^\prime$.

The group scheme $\overline{\cX}$ is unique scheme satisfying the above, up to canonical isomorphism.
\end{fact}
%

\begin{lemma}\label{L:existence of birational group law}
Let $(G,p)\in \GVar[F]$ with $G$ an algebraic group over $F$ and assume that $p\otimes p$ and $p\otimes p\otimes p$ are strictly based on $F$ and that $p^2=p$. For any affine open subvariety $V\subseteq G$, for which $\Phi(V,p)$ is of finite type over $\cO_F$, \[\Phi(m):\Phi(V,p)\times_{\cO_F}\Phi(V,p)\dashrightarrow \Phi(V,p),\] is an $\cO_F$-birational group law, where $m:V\times V\dashrightarrow V$ is the rational map given by the group multiplication on $G$.
\end{lemma}
\begin{proof}
By Corollary \ref{C::can reduce to smooth}, we may assume that $\Phi(V,p)$ is smooth and of finite type over $\cO_F$. Thus $\Phi(V,p)$ is faithfully flat over $\cO_F$ and of finite type over $\cO_F$.  

Let $m:G\times_F G\to G$, $(x,y)\mapsto xy$, be the multiplication map on $G$ and consider the isomorphisms of varieties \[\psi_1:G\times_F G\to G\times_F G,\, (x,y)\mapsto (x,xy)\] and \[\psi_2:G\times_F G\to G\times_F G,\, (x,y)\mapsto (xy,y);\] note that $m_*(p\otimes p)=p$, $(\psi_1)_*(p\otimes p)=p\otimes p$ and  $(\psi_2)_*(p\otimes p)=p\otimes p$ since $p^2=p$.

Consider $W=\psi_1^{-1}(V\times V)\cap \psi_2^{-1}(V\times V)\cap (V\times V)$, it is an open affine subvariety of $V\times V$ since $\psi_1,\psi_2$ are isomorphisms and $G\times G$ is separated (by \cite[Tag 01KU]{stacks-project}).  We now restrict $\psi_1$ and $\psi_2$ to $W$: $\psi_1:W\to V\times V$ and $\psi_2:W\to V\times V$, so they are now open immersions.

As $F[V\times V]^{p\otimes p}\cong F[V]^p \otimes_{\cO_F} F[V]^p$ (by \cite[Proposition 4.2.18]{stab-pointed}), we get that $\Phi(V\times V,p\otimes p)=\cV\times_{\cO_F} \cV$, for $\cV=\Phi(V,p)$. Thus, by applying Lemma \ref{L:locally preserves open immersions} twice, we are supplied with $f\in F[V\times V]^{p\otimes p}$ and $g_1,g_2\in F[V\times V]^{p\otimes p}$, with $p\otimes p\vdash \val(f(x))=\val(g_1(x))=\val(g_2(x))=0$, such that, for $\cU:=D_{\cV\times \cV}(f)$, \[\Phi(\psi_1):\cU\to D_{\cV\times \cV}(g_1) \text{ and }\Phi(\psi_2):\cU\to D_{\cV\times \cV}(g_2)\] are isomorphisms. Since $p\otimes p\vdash \val(f(x))=0$, $\Phi(D_{V\times V}(f),p\otimes p)=D_{\cV\times \cV}(f)$ and so is faithfully flat over $\cO_F$.

Note that $\cV\times_{\cO_F}\cV$  is smooth finitely presented and flat over $\cO_F$, and thus so is $\cU$, and all the fibers of $\cV\times_{\cO_F}\cV\to \Spec \cO_F$ are (geometrically) integral (Proposition \ref{P:each fiber is integral}). In order to conclude that $\cU$ is $\cO_F$-dense in $\cV\times_{\cO_F}\cV$ we need it, by Fact \ref{F: meets all fibers so dense}, to meet all the fibers of $\cV\times_{\cO_K}\cV\to \Spec\cO_F$. But this is automatic since $\cU$ is faithfully flat over $\cO_F$.
%

Thus $\Phi(\psi_1):\cV\times_{\cO_F}\cV\dashrightarrow \cV\times_{\cO_F}\cV$ and $\Phi(\psi_2):\cV\times_{\cO_F}\cV\dashrightarrow \cV\times_{\cO_F}\cV$ are $\cO_F$-birational by the paragraph after \cite[Definition 1, Section 5.1]{bosch} (i.e. the images $\Phi(\psi_1)(\cU)$ and $\Phi(\psi_2)(\cU)$ are also $\cO_F$-dense).

Now, by letting $\pi_i$ be the projection on the $i$-th coordinate,  we get that $\pi_2\circ\Phi(\psi_1)=\pi_1\circ \Phi(\psi_2)$. Setting $\widetilde m=\pi_2\circ\Phi(\psi_1)$ we get an $\cO_F$-rational map $\cV\times_{\cO_F}\cV\dashrightarrow \cV$ which gives an $\cO_F$-birational group law (getting associativity is easy, by further restricting $D_{V\times V}(f)$ in the beginning). 
\end{proof}

\begin{proposition}\label{P: existence of a group scheme}
Let $(G,p)\in \GVar[F]$ with $G$ an algebraic group over $F$, with $p$ satisfying $p^2=p$ and $p\otimes p$, $p\otimes p\otimes p$ strictly based on $F$. Further assume there exists an open subvariety $V\subseteq G$ for which $\Phi(V,p)$ is of finite type over $\cO_F$. Then there exists a smooth integral separated  $\cO_F$-group scheme $\cG$ of  finite type over $\cO_F$ with geometrically integral fibers and an isomorphism $G\cong \cG_F$ such that under this isomorphism $Stab(p)$ is isomorphic to $\cG(\cO)$.

Moreover, there is some $f\in F[V]$, such that $\Phi(D_V(f),p)$ is isomorphic to an open subscheme of $\cG$.
\end{proposition}
\begin{proof}
By Lemma \ref{L:existence of birational group law} and Fact \ref{F:grpchunk} there exists a smooth separated $\cO_F$-group scheme $\cG$ of finite type and faithfully flat over $\cO_F$ together with an $\cO_F$-dense open subscheme $\cU\subseteq \Phi(V,p)$ and an open immersion $\cU\to \cG$ having $\cO_F$-dense image.  Let $\cV=\Phi(V,p)$.

Since $\cU$ is $\cO_F$-dense in $\cV$, and $\cV\to \Spec \cO_F$ has geometrically integral fibers, $\cU_{\bk_F}$ is Zariski dense in $\cV_{\bk_F}$. The latter is non-empty and so $\cU_{\bk_F}$ is non-empty as well. Thus there exists an element $f\in F[V]^p$ such that $D_\cV(f)\subseteq \cU$ and $D_\cV(f)_{\bk_F}$ is non-empty (and thus $D_\cV(f)_{\bk_F}(\bk_F)$ is non empty). Replace $\cU$ by $D_\cV(f)$.

Since $\Phi(V,p)$ is an integral scheme and flat dominant over $\cO_F$ so is $\cU$. By Fact \ref{F: empty O-points implies empty k-points}, $\cU(\cO)\neq \emptyset$ so $\cU_{\cO_K}$ is faithfully flat over $\cO_K$ by \cite[Lemma 3.1.4]{stab-pointed}. By faithfully flat descent $\cU$ is faithfully flat over $\cO_F$ as well.  Furthermore, by the mmp w.r.t. $p$, $p$ is then concentrated on $\cU(\cO)$ so $p\vdash \val(f(x))=0$. Consequently, $(F[V]^p)_f=(F[V]_f)^p$, i.e. $\Phi(\cU_F,p)=\cU$.

By applying the uniqueness clause of Fact \ref{F:grpchunk} to the variety $V$ over $F$ together with the $F$-birational group law inherited from $G$, we deduce that $G$ is isomorphic to $\cG_F$ since $\cV_F=V$. We now identify between $\cG_F$ and $G$. By applying Fact \ref{F:integrality goes using flatness}, we conclude that $\cG$ is an integral scheme.
%
%

Identify $\cU$ with its image inside $\cG$.  Since $\cU\to \Spec\cO_F$ has geometrically integral fibers (by Proposition \ref{P:each fiber is integral}) and $\cU_{\bk_F}\neq \emptyset$ is an open subvariety of the group scheme $\cG_{\bk_F}$ the latter is geometrically integral (i.e. an algebraic group); likewise, for the rest of the fibers.

%

By Lemma \ref{L: p is the unique generically stable generic}, $\cG(\cO)$ is generically stable with $p$ its unique generic type; so $\mathrm{Stab}(p)=\cG(\cO)$.
\end{proof}
%
%
%
%
%

\begin{theorem}\label{T:Abelian varieties}
Let $A$ be an Abelian variety over $K$ and assume that $A$ is a generically stable group with a unique generic type $p$. Then the following are equivalent
\begin{enumerate}
\item There exists an open affine subvariety $V\subseteq A$ such that $\Phi(V,p)$ is of finite type over $\cO_K$.
\item There exists an integral Abelian scheme $\cA$ over $\cO_K$ with $A\cong \cA_K$.\footnote{$\cA$ is \emph{probably} the N\'eron model of $A$ over $\cO_K$, see the unpublished result \cite{Abelian-schemes-neron-mathoverflow}.}
\item $p$ satisfies $(\dagger)$ of Corollary \ref{C: dagger implies defectlesshenselian}.
\end{enumerate} 
\end{theorem}
\begin{proof}
$(1)\implies (2)$. By Proposition \ref{P: existence of a group scheme}, there exists a smooth integral separated group scheme $\cA$ over $\cO_K$ which is of finite type over $\cO_K$, with integral fibers, satisfying $A\cong \cA_K$. Since $\mathrm{Stab}(p)=A$ it follows that $\cA(\cO)=\cA_K$. As $A$ is an Abelian variety, it is projective  and so proper over $K$, we can thus conclude by \cite[Proposition 5.3.4]{stab-pointed} that $\cA$ is universally closed over $\cO_K$ so $\cA$ is proper over $\cO_K$. Since all the fibers of $\cA\to \Spec \cO_K$ are geometrically irreducible, $\cA$ is an Abelian scheme over $\cO_K$.

$(2)\implies (3)$. This is Proposition \ref{P: generic of Abelian scheme is strongly^2}.

$(3)\implies (1)$. This is Proposition \ref{P:strongly^2 gives finite type}. 
\end{proof}

In Section \ref{s:elliptic} we show that if $p$ is the generic type of a connected generically stable definable subgroup of an elliptic curve then it satisfies $(3)$.

\begin{question}
For $A$ and $p$ as in Theorem \ref{T:Abelian varieties}, does $p$ always satisfy $(3)$? 

More generally, if a maximal (connected) generically stable subgroup of $A$ exists, does its generic type satisfy (3)?
\end{question}
%

\section{elliptic Curves}\label{s:elliptic}
Let $F$ be a gracious valued field and $K$ a model of ACVF containing $F$.

Let $E$ be an elliptic curve over $F$, i.e. a smooth projective curve of genus $1$, which we see as a closed subscheme of $\mathbb{P}^2_F$  given by a Weierstrass equation
\[y^2z+a_1xyz+a_3yz^2=x^3+a_2x^2z+a_4xz^2+a_6z^3.\] It is automatically geometrically integral and separated \cite{elliptic-are-irreducible}. Note that $E=D_{+,E}(z)\cup D_{+,E}(y)$. The group $E(F)$ is Abelian; we denote its multiplication by $\cdot$ and by $e$ its identity element (which is the unique point at infinity). 

After a suitable change of coordinate, which yields an isomorphic copy of $E$, we may assume that $\val(a_i)\geq 0$ for $i=1,2,3,4,6$ \cite[Section VII.1]{silverman} (the assumption there that the valuation is discrete is immaterial). As $E$ is a curve, every generically stable type concentrated on $E$ is strongly stably dominated (\cite[Lemma 8.1.4(1)]{Non-Arch-Tame}) and by Lemma \ref{L:lemma 5.1-expanded}, every generically stable type concentrated on $E$ is a translate of a generic of a connected generically stable subgroup of $E$. Recall that by Fact \ref{F:strongly stably dominated groups}, every type-definable generically stable subgroup of $E$ is definable.

For the following, recall that we denote by $p_\cO$ the unique generic of the closed ball $\cO$.

\begin{lemma}\label{L:mmp in elliptic}
For any gracious valued field $L\supseteq F$, $x\models p_\cO|L$ and $y$ an element satisfying that \[y^2+a_1xy+a_3y=x^3+a_2x^2+a_4x+a_6,\]
we have that
\begin{enumerate}
\item $\val(y)\geq 0$ and $\{x^iy^j\}_{i\geq 0, j=0,1}$ constitutes a separating basis of $L[x,y]$ over $L$.
\end{enumerate}
As a result,
\begin{enumerate}
\item[(i)] The collection of formulas stating that  $x\models p_\cO$ and \[y^2+a_1xy+a_3y=x^3+a_2x^2+a_4x+a_6\] determines a unique generically stable $F$-definable type $p(x,y)$.
\item[(ii)] $p$ satisfies $(\dagger)$ of Corollary \ref{C: dagger implies defectlesshenselian}.
\item[(iii)] $p^{\otimes n}$ is strictly based on $F$, for any $n\geq 1$.
\end{enumerate}
\end{lemma}
\begin{proof}
As $x\models p_\cO|L$, $x$ constitutes a valuation transcendence basis for $L[x]$ over $L$; by \cite[Theorem 1.1]{kuhlmann} $L(x)$ is a defectless valued field. By the choice of $y$, $L(x,y)/L(x)$ is an extension of valued fields of degree $2$. As $\val(a_i)\geq 0$ for $i=2,4,6$ and $\val(x)=0$, the equation defining the elliptic curve implies that $\val(y^2+a_1xy+a_3y)\geq 0$. Using the ultrametric inequality, a simple computation now gives that necessarily $\val(y)\geq 0$. By the choice of $y$, $\bk_L(\res(x),\res(y))$ is an extension of degree $2$ of $\bk_{L(x)}=\bk_L(\res(x))$ and so $\val(y)=0$. By the fundamental inequality for extensions of valued fields we have $\bk_L(\res(x),\res(y))=\bk_{L(x,y)}$, and the result follows.


$(i)$ A  curve satisfying $y^2+a_1xy+a_3y=x^3+a_2x^2+a_4x+a_6$ must be irreducible \cite{elliptic-are-irreducible}; thus applying (1) to $L$, a model of ACVF containing $F$, we conclude by quantifier elimination in ACVF that these formulas determine a unique generically stable type (indeed, it must be orthogonal to $\Gamma$). It is clearly $F$-definable. 

$(ii)$ This is a rephrasing of $(i)$ with the finite morphism being the projection on the $x$-coordinate.
%
%

$(iii)$ This follows from $(1)$.
\end{proof}

Let $\cV$ be the closed subscheme of $\mathbb{A}^2_{\cO_F}$ (in variables $x,y$) given by $y^2+a_1xy+a_3y=x^3+a_2x^2+a_4x+a_6$, and let $V=\mathcal{V}_F=D_{+,E}(z)$.
By Lemma \ref{L:mmp in elliptic}, $\cV$ has the mmp w.r.t. $p$, where $p$ is the type from Lemma \ref{L:mmp in elliptic}. 

The following is straightforward (using Lemma \ref{L:mmp in elliptic}).

\begin{corollary}\label{C:reduction is just the fiber}
We have $\cV\cong \Phi(V,p)$ and for any $\mathfrak{p}\in \Spec \cO_F$, the fiber $\Phi(V,p)_\mathfrak{p}$ is the reduction of $\cV$ to $\bk_w$, where $w$ is the coarsening of $\val\restriction F$ corresponding to $\mathfrak{p}$.
\end{corollary}

The identity element of $E(K)$ is the only $K$-point not in $V(K)$.  The multiplication law, see e.g. \cite[Remark 3.6.1]{silverman}, on points $(x_1,y_1),(x_2,y_2)\in V$ with $(x_2,y_2)\neq (x_1,y_1)^{\pm 1}$ is given by $(x_3,y_3)=(x_1,y_1)\cdot (x_2,y_2)$, where
\[x_3=\lambda^2+a_1\lambda-a_2-x_1-x_2,\]
\[y_3=-(\lambda+a_1)x_3-\nu-a_3\]
and
\[\lambda=\frac{y_2-y_1}{x_2-x_1},\, \nu=\frac{y_1x_2-y_2x_1}{x_2-x_1}.\]
For the inverse, 
\[(x_1,y_1)^{-1}=(x_1,-y_1-a_1x_1-a_3).\]

\begin{lemma}\label{L:p is generic}
%
%
The type $p$ is the unique generic of a connected generically stable definable subgroup of $E$.
\end{lemma}
\begin{proof}
The result will follow once we show that $p^2=p$;  for then $p$ is the unique generic type of $\mathrm{Stab}(p)$. The latter is definable since $p$ is strongly stable dominated.

We show that $p^2=p$. By base-changing everything to a maximally complete\footnote{By Zorn's lemma, $K$ has an immediate maximally complete extension \cite[Theorem 8.22]{kuhlmann}, $L$; it is necessarily algebraically closed. By model completeness, $L\succ K$.} ACVF containing $K$ and since the conclusion descends, there is no harm in assuming $K$ is maximally complete.

For an element $g\in V$ we write $g=(g_x,g_y)$. Let $d\models p|K$ and $c\models p|Kd$. As $p$ is concentrated on $\cV(\cO)$, we have $\val(d_x),\val(d_y)\geq 0$. 

\begin{claim}
\begin{enumerate}
\item $\val(d_x-c_x)=0$
\item $\val(d_x-(cd^{-1})_x)=0$
\item $pp$ and $pp^{-1}$ are concentrated on $\cV(\cO)$.
\end{enumerate}
\end{claim}
\begin{claimproof}
(1) As $c\models p|Kd$ and $\val(c_x)=0$, by Lemma \ref{L:mmp in elliptic}(1) we get $\val(d_x-c_x)=\min\{\val(d_x),0\}=0$.

(2) Note that $(cd^{-1})_x=\lambda^2+a_1\lambda-a_2-c_x-d_x$,
where $\lambda=\frac{-d_y-a_1d_x-a_3-c_y}{d_x-c_x}$. Using $c_y^2=c_x^3+a_2c_x^2+a_4c_x+a_6-a_1c_xc_y-a_3c_y$ and similarly for $d_y^2$, we get
\[d_x-(cd^{-1})_x=\frac{(d_x^3-(a_1a_3+a_4)d_x-a_3d_y-2a_6-a_3^2)\cdot 1}{(d_x-c_x)^2}\]
\[+\frac{(-a_4-a_1d_y-a_1^2d_x-a_1a_3-2a_2d_x-3d_x^2)\cdot c_x+(-2d_y-a_1d_x-a_3)\cdot c_y}{(d_x-c_x)^2}.\]

By Lemma \ref{L:mmp in elliptic}(1) and (1), and since $c\models p|Kd$,
\[\val(d_x-(cd^{-1})_x)=\min\{\val(d_x^3-(a_1a_3+a_4)d_x-a_3d_y-2a_6-a_3^2),\]\[ \val((-a_4-a_1d_y-a_1^2d_x-a_1a_3-2a_2d_x-3d_x^2),\val(-2d_y-a_1d_x-a_3)\}.\]

Note that all the elements have non-negative valuation. Since $d\models p|K$ and using Lemma \ref{L:mmp in elliptic}(1), again, we conclude that $\val(d_x-(cd^{-1})_x)=0$.

(3) Since $p$ is concentrated on $\cV(\cO)$, so is $pp$ by $(1)$ and the formulas for multiplication; similarly, so is $pp^{-1}$.
\end{claimproof}

Let $f\in K[V]$ be some regular function. By the mmp of $\cV$ w.r.t. $p$ and item $(3)$ of the claim, $\val(f(c))\leq \val(f(cd))$, we now show the other direction; since $f$ is arbitrary this will give the desired conclusion that $pp=p$.

By the multiplication law on $E$, there exist $f_i,g_i\in K[V]$ and some integer $n$ such that for any $a,b\in V$, with $a\neq b^{\pm 1}$, we have 
\[f(a\cdot b)=\frac{\sum_i f_i(a)g_i(b)}{(b_x-a_x)^n}.\]

As $K$ is maximally complete and $(c,d)\models (p\otimes p)|K$, by \cite[Lemma 12.4]{StableDomination} (see also \cite[Proposition 4.2.17]{stab-pointed}), $f_i$ and $g_i$ may be chosen such that $\val(\sum_if_i(c)g_i(d))=\min_i\{\val(f_i(c))\}$ and $\val(g_i(d))=0$.

We can now use the above claim and the fact that $\val(f_i(c))\leq \val(f_i(cd^{-1}))$ (by the mmp of $\cV$ w.r.t. $p$ and the fact that $pp^{-1}$ is concentrated on $\cV(\cO)$) to compute
\[\val(f(cd))=\val\left(\frac{\sum_if_i(c)g_i(d)}{(d_x-c_x)^n}\right)=\val\left(\sum_if_i(c)g_i(d)\right)=\]
\[=\min_i\{\val(f_i(c))\}\leq \min_i\{\val(f_i(cd^{-1}))\}=\min_i\{\val(f_i(cd^{-1})g_i(d))\}\]
\[\leq \val\left(\sum_if_i(cd^{-1})g_i(d)\right)= \val\left(\frac{\sum_if_i(cd^{-1})g_i(d)}{(d_x-(cd^{-1})_x)^n}\right)=\val(f(c)),\]
as required.
\end{proof}

\subsection{The family of generics}\label{ss: family of generics}
We keep the notation from before but now assume that $\mathrm{char}(\bk_F)\neq 2,3$. Making this assumption eases the computations to come; it would be interesting to verify that they hold in general.

Given that $\mathrm{char}(F)\neq 2,3$, any elliptic curve is isomorphic to one with a Weierstrass equation of the form
\[y^2z=x^3+Axz^2+Bz^3,\]
with $\val(A),\val(B)\geq 0$  \cite[Section III.1]{silverman}. The discriminant of this equation is equal to $\Delta=-16(4A^3+27B^2)$.

In this case the multiplication law specializes to the following: for points $(x_1,y_1),(x_2,y_2)\in V$ with $(x_2,y_2)\neq (x_1,y_1)^{\pm 1}$ we get that $(x_3,y_3)=(x_1,y_1)\cdot (x_2,y_2)$, where
\[x_3=\lambda^2-x_1-x_2,\]
\[y_3=-\lambda x_3-\nu\]
and
\[\lambda=\frac{y_2-y_1}{x_2-x_1},\, \nu=\frac{y_1x_2-y_2x_1}{x_2-x_1}.\]
For the inverse, 
\[(x_1,y_1)^{-1}=(x_1,-y_1).\]

It is worthwhile to expand the formula for multiplication, after plugging in $y_1^2=x_1^3+Ax_1+B$ and $y_2^2=x_2^3+Ax_2+B$ we get
\begin{equation}\label{eq: for multiplication}
x_3=\frac{(Ax_2+2B)+(x_2^2+A)\cdot x_1+(-2y_2)\cdot y_1+x_2\cdot x_1^2}{x_2^2-2x_2x_1+x_1^2}.
\end{equation}


\begin{notation}
For any type $q$ concentrated on $V$ let $q_x$ be the type $\pi_*q$, where $\pi:V\to \mathbb{A}^1_K$ is the projection on the first coordinate. 
\end{notation}

\begin{lemma}\label{L:existence of p_a}
	For every $a\in F^\times $ with $\val(a)\leq\min\{\frac{1}{2}\val(A),\frac{1}{3}\val(B)\}$, there exists a unique generically stable $F$-definable type $p_a$ concentrated on $V$ satisfying that $(p_a)_x=ap_\cO$.
	
	For any gracious valued field $L\supseteq F$ and $d=(d_x,d_y)\models (p_a)|L$, \[\left\{\left(\frac{d_x}{a}\right)^i\cdot \left(\frac{d_y}{a^{3/2}}\right)^j\right\}_{i\geq 0,\, j=0,1}\] constitutes a separating  basis of $L[\frac{d_x}{a},\frac{d_y}{a^{3/2}}]$ over $L$.
	
	As a result, $p_a$ satisfies $(\dagger)$ of Corollary \ref{C: dagger implies defectlesshenselian} and $p_a^{\otimes n}$ is strictly based on $F$, for any $n\geq 1$.
\end{lemma}
\begin{remark}
Naturally, $a^{3/2}$ is not well-defined for an element $a\in K^\times$. However, any choice of square-root would work here. 
\end{remark}
\begin{proof}
We apply Lemma \ref{L:mmp in elliptic} to the Weierstrass equation $y^2=x^3+\frac{A}{a^2}x+\frac{B}{a^3}$; note that by the assumptions, $\val\left(\frac{A}{a^2}\right)\geq 0$ and $\val\left(\frac{B}{a^3}\right)\geq 0$ and that its discriminant is non-zero. Hence, its projective closure defines an elliptic curve. Let $\widehat p_a$ be the unique generically stable $L$-definable type we are supplied with and let $p_a$ be the type of $(a\widehat d_x,a^{3/2}\widehat d_y)$ over $\mathbb{K}$, for $\widehat d=(\widehat d_x,\widehat d_y)\models \widehat p_a$.

To finish, note that if $d=(d_x,d_y)\models p_a|L$ then $(p_a)_x=ap_\cO$, and if $u\models p_\cO|L$ satisfies that $d_x=au$ then, plugging this into $d_y^2=d_x^3+Ad_x+B$ and dividing by $a^3$ yields:
\[\left(\frac{d_y}{a^{3/2}}\right)^2=u^3+\frac{A}{a^2}u+\frac{B}{u^3}=\left(\frac{d_x}{a}\right)^3+\frac{A}{a^2}\left(\frac{d_x}{a}\right)+\frac{B}{a^3}.\] 

The rest follows by (and is as in) Lemma \ref{L:mmp in elliptic}.
\end{proof}

\begin{remark}\label{R:pgamma}
We note the following.
\begin{enumerate}
\item $p_a$ concentrates on $\cV(\cO)$ if and only if $\val(a)\geq 0$.
\item By uniqueness,  $\val(a_1)=\val(a_2)$ if and only if $p_{a_1}=p_{a_2}$.
\end{enumerate}
\end{remark}

\begin{definition}
Given a Weierstrass equation of the form $y^2=x^3+Ax+B$, with $\val(A),\val(B)\geq 0$, we set $\gamma_\infty=\min\{\frac{1}{2}\val(A),\frac{1}{3}\val(B)\}$ and for any $\gamma\leq \gamma_\infty$ we let $p_\gamma=p_a$ for any $a\in K^\times$ with $\val(a)=\gamma$.
\end{definition}

The following lemma gives that every generic of a connected generically stable definable subgroup of $E$ is of the form $p_\gamma$ for some $\gamma$.

\begin{lemma}\label{L: classifying generics}
If $q$ is a non-algebraic generically stable $K$-definable type satisfying $q^2=q$ then $q=p_\gamma$ for some $\gamma\leq \gamma_\infty$.

If $q$ is $F$-definable and strictly based on $F$ then we may find such $\gamma\in \Gamma_F$.
\end{lemma}
\begin{proof}
Let $q$ be a non-algebraic  generically stable $K$-definable satisfying $q^2=q$; so $q$ is not the type of the identity element.

Since every generically stable $K$-definable type in the valued field sort is the generic type of a closed ball over $K$ (\cite[Lemma 2.5.5]{Eli_Imag}), necessarily $q_x=ap_\cO+b$, for some $a,b\in K$ with $a\neq 0$.


Let $d=(d_x,d_y)\models q|K$ and $f=(f_x,f_y)\models q|Kd$. Set $(x_3,y_3)=(d_x,d_y)\cdot (f_x,f_y)$ and \[(x_3',y_3')=(d_x,d_y)\cdot (f_x,f_y)^{-1}=(d_x,d_y)\cdot (f_x,-f_y).\]

By our assumption that $q^2=q$ we conclude that $q=q\cdot q=q\cdot q^{-1}$. Thus both $x_3$ and $x_3'$ (in addition to $f_x$ and $d_x$) satisfy $q_x|K$.
Using the formulas for multiplication, \[-2f_yd_y=(f_x-d_x)^2(x_3-x_3').\] Note that for any $r,s\models q_x|K$, $\val(r-s)\geq \val(a)$; indeed, writing $r=au+b$ and $s=av+b$ for $u,v\models p_\cO|K$, $\val(r-s)=\val(a(u-v))\geq \val(a)$. Since $q$ is generically stable, $\val(d_y)=\val(f_y)\in \Gamma_K$ so 
\[\val(-2f_yd_y)=2\val(d_y)=2\val(f_x-d_x)+\val(x_3-x_3')\geq 3\val(a),\]
that is, $\val(d_y)\geq \frac{3}{2}\val(a)$.

Setting $d_x=au+b$ for $u\models p_\cO|K$ in the equation $d_y^2=d_x^3+Ad_x+B$ yields

\[d_y^2=a^3u^3+3a^2bu^2+(3ab^2+Aa)u+(b^3+Ab+B).\]

By Lemma \ref{L:mmp in elliptic}(1),
\begin{equation}\label{eq:equation}
2\val(d_y)=\min\{3\val(a),2\val(a)+\val(b),\val(a)+\val(3b^2+A),\val(b^3+Ab+B)\}.
\end{equation}

Hence $2\val(d_y)\leq 3\val(a)$; we conclude that $\val(d_y)= \frac{3}{2}\val(a)$. Also $\val(d_y)\leq \val(a)+\frac{1}{2}\val(b)$ so $\val(b)\geq \val(a)$, i.e. $\val(\frac{b}{a})\geq 0$ giving that $ap_\cO+b=ap_\cO$. As a result, in Equation \ref{eq:equation} there is no harm in taking $b=0$.  In particular $2\val(d_y)=3\val(a)\leq \val(a)+\val(A)$, so $\val(a)\leq \frac{\val(A)}{2}$ and similarly $\val(a)\leq \frac{\val(B)}{3}$.

If $q$ is $F$-definable and strictly based on $F$ then $\val(d_x)\in \Gamma_F$ so we may choose $a\in F^\times$ with $q=p_a$.
%
%
%
%
%
%
%
%
\end{proof}

We now show that these types are exactly the generic types of connected generically stable definable subgroups of $E$.

\begin{lemma}\label{L: product of generics}
For any $\gamma_1,\gamma_2\in \Gamma_K$ satisfying $\gamma_1\leq \gamma_2\leq \gamma_\infty$,  $p_{\gamma_1}p_{\gamma_2}=p_{\gamma_2}$.

In particular, for every such $\gamma$, $p_\gamma$ is the unique generic of a connected generically stable definable  subgroup of $E$.
\end{lemma}
\begin{proof}
Choose some $a_1,a_2\in K^\times$ with $\val(a_1)=\gamma_1$ and $\val(a_2)=\gamma_2$.

We first assume that $\gamma_1=\gamma_2$. There is no harm in assuming that $a:=a_1=a_2$. By Lemma \ref{L:existence of p_a}, it is sufficient to prove that $\frac{1}{a}(p_a\cdot p_a)_x=p_\cO$.

Let $d=(d_x,d_y)\models p_a|K$ and $f=(f_x,f_y)\models p_a|Kd$. Consider the elliptic curve $E_a$ given by the Weierstrass equation $y^2=x^3+\frac{A}{a^2}x+\frac{B}{a^3}$ and let $\widehat p$ be the type Lemma \ref{L:mmp in elliptic} supplies $E_a$; note that $(\widehat p)_x=p_\cO$. Let $*$ be the multiplication on $E_a$ and $\cdot$ be that on $E$. 

Quick computations gives that $\widehat d=(\frac{d_x}{a},\frac{d_y}{a^{3/2}})\models \widehat p|K$ and $\widehat f=(\frac{f_x}{a},\frac{f_y}{a^{3/2}})\models \widehat p|K\widehat d$ (see the proof of Lemma \ref{L:existence of p_a}). By the formulas for multiplication on $E$ and $E_a$,
\[\frac{1}{a}(f\cdot d)_x=\frac{(d_y-f_y)^2}{a(d_x-f_x)^2}-\frac{d_x}{a}-\frac{f_x}{a}= \frac{a^3(\frac{d_y}{a^{3/2}}-\frac{f_y}{a^{3/2}})^2}{a^3(\frac{d_x}{a}-\frac{f_x}{a})^2}-\frac{d_x}{a}-\frac{f_x}{a}=(\widehat f *\widehat d)_x.\] 

Since $(\widehat p)^2=\widehat p$ (Lemma \ref{L:p is generic} applied to $E_a$),  $\frac{1}{a}(p_a\cdot p_a)_x=(\widehat p* \widehat p)_x=\widehat p_x=p_\cO$. It follows that $(p_a)^2=p_a$ and so $p_a$ is the unique generic type of $\mathrm{Stab}(p_a)$.

We now assume that $\gamma_1=\val(a_1)<\val(a_2)=\gamma_2$. Since $E$ is Abelian, and by using the above, $(p_{a_1}p_{a_2})^2=p_{a_1}p_{a_2}$. By Lemma \ref{L: classifying generics} there is some $a_3\in K^\times$ with $p_{a_1}p_{a_2}=p_{a_3}$. To show that $p_{a_3}=p_{a_2}$, we will show that $\val(a_3)=\val(a_2)$, for then $(p_{a_1}p_{a_2})_x=(p_{a_2})_x$ (by Remark \ref{R:pgamma}(2)).

Let $d=(d_x,d_y)\models p_{a_2}|K$ and $f=(f_x,f_y)\models p_{a_1}|Kd$. By Equation \ref{eq: for multiplication}, 
\[(fd)_x=\frac{(Ad_x+2B)\cdot 1+(d_x^2+A)\cdot f_x+(-2d_y)\cdot f_y+d_x\cdot f_x^2}{d_x^2-2d_xf_x+f_x^2}\]
\[=\frac{\left(Aa_2\left(\frac{d_x}{a_2}\right)+2B\right)\cdot 1+\left(a_2^2a_1\left(\frac{d_x}{a_2}\right)^2+Aa_1\right)\cdot \left(\frac{f_x}{a_1}\right)+\left(-2a_1^{3/2}a_2^{3/2}\left(\frac{d_y}{a_2^{3/2}}\right)\right)\cdot \left(\frac{f_y}{a_1^{3/2}}\right)}{a_2^2\left(\frac{d_x}{a_2}\right)^2-2a_1a_2\left(\frac{d_x}{a_2}\right)\left(\frac{f_x}{a_1}\right)+a_1^2\left(\frac{f_x}{a_1}\right)^2}\]\[+\frac{\left(a_1^2a_2\left(\frac{d_x}{a_2}\right)\right)\cdot \left(\frac{f_x}{a_1}\right)^2}{a_2^2\left(\frac{d_x}{a_2}\right)^2-2a_1a_2\left(\frac{d_x}{a_2}\right)\left(\frac{f_x}{a_1}\right)+a_1^2\left(\frac{f_x}{a_1}\right)^2}.\]

Applying Lemma \ref{L:existence of p_a} (twice) to the denominator, we get that, since $\val(f_x)=\val(a_1)<\val(a_2)=\val(d_x)$, its value is equal to $\min\{2\val(a_2),\val(a_1)+\val(a_2),2\val(a_1)\}=2\val(a_1)$.

On the other hand, since $\val(a_1)< \val(a_2)\leq\min\{\frac{1}{2}\val(A),\frac{1}{3}\val(B)\}$ and using Lemma \ref{L:existence of p_a} (twice), we get that the value of nominator is equal to $\val(a_2)+2\val(a_1)$. In conclusion, $\val((fd)_x)=\val(a_2)$, i.e. $\val(a_3)=\val(a_2)$, so $p_{a_1}p_{a_2}=p_{a_2}$.
\end{proof}

\begin{proposition}\label{P:maximal gen substable subgroup}
Let $E$ be an elliptic curve over $F$, with $\mathrm{char}(\bk_F)\neq 2,3$.

Then $E$ contains a maximal connected generically stable (type-)definable subgroup. We name its unique generic type $p_\infty$.

If $E$ is given by a Weierstrass equation $y^2=x^3+Ax+B$ with $\val(A),\val(B)\geq 0$ then $p_\infty=p_{\gamma_\infty}$.
\end{proposition}
\begin{proof}

There is no harm in assuming that $E$ is given by a Weierstrass equation of the form $y^2=x^3+Ax+B$ with $\val(A),\val(B)\geq 0$. We set $p_\infty=p_{\gamma_\infty}$.

Let $q$ be a non-realized generically stable $K$-definable type concentrated on $E$ satisfying $q^2=q$. By Lemma \ref{L: classifying generics}, $q=p_{\gamma}$ for some $\gamma\leq \gamma_\infty$. By Lemma \ref{L: product of generics}, $qp_\infty=p_\infty$ and hence $\mathrm{Stab}(q)\subseteq \mathrm{Stab}(p_\infty)$.
\end{proof}


\begin{remark}
By \cite[Corollary 6.19]{Metastable}, there exists a (necessarily unique) $K$-definable generically stable type $p_\infty$ concentrated on $E$ satisfying that $p_\infty^2=p_\infty$ and that $qp_\infty=p_\infty$ for any other generically stable $K$-definable type $q$ concentrated on $E$ satisfying $q^2=q$, i.e. a generic type of a generically stable subgroup. Proposition \ref{P:maximal gen substable subgroup} gives a direct proof of this fact. Moreover, we exhibit $\mathrm{Stab}(p_\infty)$  as a limit of generically stable subgroups of $E$, giving \cite[Lemma 6.18]{Metastable} in this situation
\end{remark}

We end this section with a nice consequence of the above results.
\begin{proposition}
Let $E$ be an elliptic curve over $K$, with $\mathrm{char}(\bk_K)\neq 2,3$, and $q$ a $K$-definable generically stable non-algebraic type concentrated on $E$. Then there exists an open affine subvariety $U\subseteq E$ with $\Phi(U,q)$ of finite type over $\cO_K$.
\end{proposition}
\begin{proof}
There is no harm in assuming that $E$ is given by $y^2=x^3+Ax+B$ with $\val(A),\val(B)\geq 0$. By Lemma \ref{L:lemma 5.1-expanded} there exists $g\in E$ with $gq$ satisfying $(gq)^2=gq$. By Lemma \ref{L: classifying generics}, there is some $\gamma\leq \gamma_\infty$ with $gq=p_\gamma$. 

Applying Lemma \ref{L:existence of p_a} (and Corollary \ref{C: dagger implies defectlesshenselian}), there is some affine open subvariety $V\subseteq E$ such that $\Phi(V,p_\gamma)$ is of finite type over $\cO_K$. As $(V,p_\gamma)$ and $(g^{-1}V,q)$ are isomorphic in $\aGVar$, $\Phi(V,p_\gamma)$ is isomorphic to $\Phi(g^{-1}V,q)$. Hence, setting $U=g^{-1}V$, we get the desired result.
\end{proof}

\subsection{Getting group schemes}
The purpose of this final section is to translate the results in the previous sections to a geometric language. Unless stated otherwise, we return to the setting of general Weierstrass equations and general gracious fields (i.e. without any restriction on the characteristic of the residue field).

We  require the following folklore result.

\begin{fact}\label{F:proper fibers give proper}\cite[Corollary A.3.4]{abelianschemes}
Let $f:\cX\to \cS$ be a flat, separated morphism of finite presentation. If $f$ has proper and geometrically connected fibers then it is proper.
\end{fact}

Recall that the discriminant of a Weierstrass equation $y^2+a_1xy+a_3y=x^3+a_2x^2+a_4x+a_6$  is given by  $\Delta=-b_2^2b_8-8b_4^3-27b^2_6+9b_2b_4b_6$, where $b_2=a_1^2+4a_2$, $b_4=2a_4+a_1a_3$, $b_6=a_3^2+4a_6$ and $b_8=a_1^2a_6+4a_2a_6-a_1a_3a_4+a_2a_3^2-a_4^2$. The projective closure of the curve it defines is an elliptic curve (i.e. smooth) if and only if  $\Delta\neq 0$ \cite[Proposition III.1.4]{silverman}.

Below, let $p$ be the generically stable type supplied by Lemma \ref{L:mmp in elliptic}.

\begin{theorem}\label{T:good reduction for elliptic curves}
Let $E$ be an elliptic curve over $F$. There exists a  smooth integral separated $\cO_F$-group scheme $\cE$ of finite type over $\cO_F$ with geometrically integral fibers, satisfying $E\cong \cE_F$ such that under this isomorphism $\mathrm{Stab}(p)=\cE(\cO)$. The following are equivalent:
\begin{enumerate}
\item $\mathrm{Stab}(p)=E$.
\item $\cE$ is proper over $\cO_F$, and thus an Abelian scheme over $\cO_F$.
\item (Assuming $E$ is given by a Weiersrtrass  equation over $\cO_F$) $\val(\Delta)=0$, where $\Delta$ is the discriminant.
\end{enumerate}

As a result, there exists a proper integral group scheme $\cG$ over $\cO_F$ with $\cG_F\cong E$ if and only if $\mathrm{Stab}(p)=\cG_F$ if and only if $\val(\Delta)=0$ (when $E$ is given by a Weiersrtrass  equation over $\cO_F$).
\end{theorem}
\begin{remark}
In the ``as a result'', we can replace ``there exists a proper integral group scheme $\cG$ over $\cO_F$ with $\cG_F\cong E$'' with ``there exists a proper irreducible smooth curve with geometrically connected fibers all of genus one, with a given section in $\cG(\cO_F)$, satisfying $\cG_F\cong E$'', since any such scheme must necessarily be a group scheme (by \cite[Theorem 2.1.2]{KaMa}).
\end{remark}
\begin{proof}
We may assume that $E$ is given by (the projective closure of) a  Weierstrass equation $y^2+a_1xy+a_3y=x^3+a_2x^2+a_4x+a_6$, with \[\val(a_1),\val(a_3),\val(a_2),\val(a_4),\val(a_6)\geq 0.\] 
By Lemma \ref{L:mmp in elliptic}(ii), $p$ satisfies $(\dagger)$ of Corollary \ref{C: dagger implies defectlesshenselian} so the existence of $\cE$ is given by Proposition \ref{P: existence of a group scheme} (and Lemma \ref{L:mmp in elliptic}(iii)).

$(1)\iff(2)$. If $\mathrm{Stab}(p)=E$ then,  by \cite[Proposition 5.3.4]{stab-pointed}, $\cE$ is universally closed over $\cO_F$, so $\cE$ is proper over $\cO_F$. As the fibers are geometrically integral, $\cE$ is an Abelian scheme over $\cO_F$. For the other direction, if $\cE$ is proper over $\cO_F$ then $\mathrm{Stab}(p)=\cE(\cO)=\cE_F=E$.

$(2)\iff (3)$. Since the assumption is preserved and the conclusion descends, there is no harm in base changing and assuming that $F=K$. 
We first make some observations. As $\cE$ is smooth over $\cO_K$, for any $\mathfrak{p}\in \Spec \cO_K$, the fiber $\cE_\mathfrak{p}$ is an algebraic group of dimension $1$ over $\kappa(\mathfrak{p})$, i.e. is either $\mathbb{G}_a$, $\mathbb{G}_m$ or an elliptic curve (note that $\mathbb{G}_a$ and $\mathbb{G}_m$ are not birational to an elliptic curve).

On the other hand, for the open affine subvariety $V\subseteq E$ from before, we have that $\Phi(V,p)$ is birational with $\cE$ (the ``moreover'' clause of Proposition \ref{P: existence of a group scheme}). So for any $\mathfrak{p}\in \Spec \cO_K$, $\cE_\mathfrak{p}$ is birational to $\Phi(V,p)_\mathfrak{p}$ which is just an open subvariety of the reduction of $\cV$ to $\kappa(\mathfrak{p})$ (see Corollary \ref{C:reduction is just the fiber}). Let $w$ be the coarsening of $\val$ corresponding to $\mathfrak{p}$; so $\cE_\mathfrak{p}$ is an elliptic curve if and only if $w(\Delta)=0$.

If $\cE$ is an Abelian scheme over $\cO_K$ then the special fiber must be an elliptic curve so $\val(\Delta)=0$. If on the other hand, $\val(\Delta)=0$ then $w(\Delta)=0$ for any coarsening $w$ of $\val$, hence $\cE_{\mathfrak{p}_w}$ is proper over $\bk_w$ for any such coarsening. Since all the fibers are (geometrically) integral they are all Abelian varieties. By Fact \ref{F:proper fibers give proper}, $\cE$ is proper over $\cO_K$, i.e. it is an Abelian scheme over $\cO_K$.

For the final statement. The statement ``$\mathrm{Stab}(p)=E$ is equivalent to $\val(\Delta)=0$'' is given by the above arguments, so we show they are equivalent to the middle statement.  If $\mathrm{Stab}(p)=E$ then $\cE$ is a proper group scheme over $\cO_F$ with $\cE_F\cong E$ (by (2)). For the other direction, there is no harm in assuming that $F=K$. Let $\cG$ be a proper integral group scheme over $\cO_K$ for which $\cG_K=E$. By \cite[Proposition 5.1.6]{stab-pointed}, $\cG(\cO)=\cG_K=E$ is a generically stable group say with principal generic $q$. As $\mathrm{Stab}(q)$ is an intersection of definable subgroups of $E$ of finite index, \cite{Metastable}, and $E$ is a divisible group, $\mathrm{Stab}(q)=E$; thus $q=p$ as needed.
\end{proof}

\begin{remark}
Let $E$ be an elliptic curve over a fraction field over a DVR $R$ given by a Weierstrass equation over $R$. This same equation defines a closed subscheme of $\mathbb{P}^2_R$. It is known that the smooth locus of this scheme carries the structure of a smooth group scheme whose generic fiber is isomorphic to $E$ \cite[Theorem IV.5.3]{Si94}. It will be interesting to check if similar arguments give a similar result over a general valued field. Our motivation here was to show that the techniques of the previous sections have applications for some Abelian varieties.
\end{remark}

\begin{remark}
It is well known that an elliptic curve over a fraction field of a  DVR has a N\'eron model; this model is in general not proper. It is well known  that an elliptic curve $E$ over a DVR has a proper N\'eron model if and only if it has good reduction (i.e. has some proper model). Given \cite{Abelian-schemes-neron-mathoverflow}, the theorem above can be seen as a generalization of this latter fact. 
\end{remark}

\begin{remark}\label{R: neron model of elliptic over acvf not exists}
It is well known that an elliptic curve over an algebraically closed valued field might not have a N\'eron model, see e.g. \cite[The comment by user gdb]{Abelian-schemes-neron-mathoverflow}. 

Another argument is provided by \cite[Proposition 5.3.4]{stab-pointed}. Indeed, over a model of ACVF this cited result shows that a N\'eron model must be proper. So the existence of a N\'eron model would imply, using Theorem \ref{T:good reduction for elliptic curves}, that $\val(\Delta)=0$. But obviously this does not hold for all elliptic curves.
\end{remark}

Nonetheless, there is a connection to the N\'eron model of an elliptic curve.

\begin{proposition}\label{P:neron}
Let $F$ be a gracious valued field with discrete value group and let $E$ be an elliptic curve over $F$ given by a minimal Weierstrass equation over $\cO_F$.\footnote{A Weierstrass equation with minimal non-negative $\val(\Delta)\in \Gamma_F$.} Let $\mathcal{N}$ be the N\'eron model of $E$, $\mathcal{N}^0$ its identity component\footnote{The open subscheme of $\mathcal{N}$ obtained by discarding the non-identity components of the special fiber.} and let $\cE$ be the group scheme over $\cO_F$ supplied by Theorem \ref{T:good reduction for elliptic curves}. Then $\cE\cong \mathcal{N}^0$.
\end{proposition}
\begin{proof}
Assume that $E$ is given by a minimal Weierstrass equation over $\cO_F$ and let $p$ be the generically stable type from Lemma \ref{L:mmp in elliptic}. By Corollary \ref{C:reduction is just the fiber}, $\cV=\Phi(V,p)$. Hence, if we let $\cW\subseteq \mathbb{P}^2_{\cO_F}$ be the closed subscheme the Weierstrass equation defines and $\cW^0$ its smooth locus then by the construction of $\cE$ (i.e. the proof of Theorem \ref{T:good reduction for elliptic curves}), $\cW^0$ is $\cO_F$-birational to $\cE$. Since $\cW^0\cong \mathcal{N}^0$ (\cite[Corollary IV.9.1]{Si94}), the uniqueness clause of Fact \ref{F:grpchunk} implies that $\cE\cong \mathcal{N}^0$ as well.
\end{proof}

%
%

We end with the following geometric interpretation of Proposition \ref{P:maximal gen substable subgroup}; for this we assume that $\mathrm{char}(\bk_F)\neq 2,3$. Recall the notation and types $p_\gamma$ from Section \ref{ss: family of generics}.

%
%

\begin{theorem}\label{T:geo interp}
Let $E$ be an elliptic curve over $F$, with $\mathrm{char}(\bk_F)\neq 2,3$, given by a Weierstrass equation $y^2=x^3+Ax+B$ with $\val(A),\val(B)\geq 0$. Let $\gamma_\infty=\min\{\frac{1}{2}\val(A),\frac{1}{3}\val(B)\}$.

For any $\gamma\in \Gamma_F$ with $\gamma\leq \gamma_\infty$, there exists a smooth integral separated $\cO_F$-group scheme $\cE_\gamma$ of finite type over $\cO_F$ with geometrically integral fibers, satisfying $E\cong (\cE_\gamma)_F$ and under this isomorphicm $\mathrm{Stab}(p_\gamma)=\cE_\gamma(\cO)$.

\begin{enumerate}
\item For any $\gamma_1,\gamma_2\in \Gamma_F$ with $\gamma_1< \gamma_2\leq \gamma_\infty$, we have $\cE_{\gamma_1}(\cO)\subsetneq \cE_{\gamma_2}(\cO)$.
\item Applying the above for $F=K$, for any integral group scheme $\cG$ of finite type over $\cO_K$ with integral special fiber and satisfying $\cG_K\cong E_K$, we have $\cG(\cO)=\cE_{\gamma}(\cO)$ for some $\gamma\in \Gamma_K$ with $\gamma\leq \gamma_\infty$ and
\item For any connected generically stable definable subgroup $H\leq G$  there is some $\gamma\leq \gamma_\infty$ with $H=\cE_{\gamma}(\cO)$. 
\end{enumerate}
\end{theorem}
\begin{proof}
Let $\gamma\in \Gamma_F$ satisfy $\gamma\leq \gamma_\infty$ and let $p_\gamma$ be the corresponding generically stable generic of a connected definable subgroup of $E$  given by Lemma \ref{L: product of generics}. By Lemma \ref{L:existence of p_a}, $p_\gamma$ satisfies $(\dagger)$ of Corollary \ref{C: dagger implies defectlesshenselian} and $p_\gamma^{\otimes n}$ is strictly based on $F$ for $n\geq 1$, so the existence of $\cE_\gamma$ is given by Proposition \ref{P: existence of a group scheme}. Moreover, the isomorphism $E\cong (\cE_\gamma)_F$ maps $\mathrm{Stab}(p_\gamma)$ onto $\cE_\gamma(\cO)$.

(1) By Lemma \ref{L: product of generics}, if $\gamma_1<\gamma_2\leq \gamma_\infty$ then $\cE_{\gamma_1}(\cO)=\mathrm{Stab}(p_{\gamma_1})\subsetneq \mathrm{Stab}(p_{\gamma_2})=\mathcal{E}_{\gamma_2}(\cO)$.

(2) By Lemma \ref{L: p is the unique generically stable generic}, $\cG(\cO)$ is a connected generically stable group with a unique generic type $q$, i.e $\mathrm{Stab}(q)=\cG(\cO)$. By Lemma \ref{L: classifying generics}, $q=p_\gamma$ for some $\gamma\in \Gamma_K$ with $\gamma\leq \gamma_\infty$. Thus $\cE_{\gamma}(\cO)=\mathrm{Stab}(q)=\cG(\cO)$.

(3) The proof is similar to (2).
\end{proof}

\begin{remark}
By model completeness,  for any valued field extension $(K,\cO_K)\leq (L,\cO_L)$, in (1) we have $(\cE_{\gamma_1})_{\cO_L}(\cO_L)\subseteq (\cE_{\gamma_2})_{\cO_L}(\cO_L)$ and in (2) we have $\cG_{\cO_L}(\cO_L)=(\cE_{\gamma})_{\cO_L}(\cO_L)$.
\end{remark}

\begin{question}
	In Theorem \ref{T:geo interp}(2), can we force, maybe after some more reasonable assumptions on $\cG$, that $\cG\cong \cE_{\gamma}$?
\end{question}

\begin{remark}
Let $E$ be the elliptic curve over $\mathbb{Q}_5$ given by the equation $y^2=x^3+5^3x+5^6$. It is a minimal Weierstrass equation by \cite[Section VII.1]{silverman}. By Proposition \ref{P:neron}, the identity component of its N\'eron model corresponds to $p=p_0$, however its $\cO$-points are not the maximal generically stable subgroup. Indeed, it corresponds to $p_{3/2}$.
\end{remark}

\appendix
\section{Generically stable types and the Shelah expansion}\label{S:generically stable and shelah}
Let $T$ be a complete NIP theory in a first order language $\mathcal{L}$ and let $\mathbb{U}$ be a large saturated model.

\begin{definition}
A global type $p$ is said to be \emph{generically stable} if it is definable and finitely satisfiable in some small model $M$. We then say that $p$ is \emph{generically stable over $M$}.
\end{definition}

\begin{remark}
\begin{enumerate}
\item For every set $B$ for which a generically stable type $p$ is $B$-invariant, $p$ is the unique $B$-invariant extension of $p|B$.
\item A global $A$-invariant type is generically stable if and only if every Morley sequence of $p$ is totally indiscernible if and only if there exists a Morley sequence of $p$ which is totally indiscernible.
\end{enumerate}
\end{remark}

\begin{fact}\label{F:fact-genstable}

\begin{enumerate}
\item Let $p$ be a generically stable type. Then $p$ is still generically stable in any reduct.
\item There is a unique expansion of any type $p$ to $T^{eq}$. If $p$ is generically stable then this expansion is still generically stable.
\end{enumerate}
\end{fact}
\begin{proof}
$(1)$ Assume that $p$ is definable and finitely satisfiable in some small model $M$. Let $\tilde p$ be the reduct of the type. Since $\tilde p$ is still finitely satisfiable in $M$ it is $M$-invariant. Now, we use the fact that every Morley sequence of $p$ is totally indiscernible to show that there exists such a Morley sequence in $\tilde p$.

$(2)$ Straightforward.
\end{proof}

Let $M\models T$ be a small model and $M\prec N$ an $|M|^+$-saturated model. The Shelah expansion of $M$, denoted by $M^{sh}$, is a structure with universe $M$ and language $\mathcal{L}^{sh}=\mathcal{L}\cup\{R_{\varphi(x,c)}: \varphi(x,y)\in \mathcal{L}, c\in N\}$, where for every $a\in M$, $M^{sh}\models R_{\varphi(x,c)}(a)\iff N\models \varphi(a,c)$. The definable sets of $M^{sh}$ do not depend on the $|M|^+$-saturated model from which we take the parameters for the externally definable sets. Since $T$ is NIP, $\mathrm{Th}(M^{sh})$ eliminates quantifiers in $\mathcal{L}^{sh}$ \cite{shelah-qe}.

By resplendency of saturated models, we may expand $\mathbb{U}$  to a saturated model of $\mathrm{Th}(M^{sh})$ in the language $\mathcal{L}^{sh}$. We denote this expansion by $\mathbb{U}^*$.

The following is probably well known but we could not find a proof of it anywhere, a version for measures appears in \cite[Theorem 6.2]{ShDistal}.

\begin{proposition}\label{P:gen-stable-sh}
Let $p$ be a global type. If $p$ is definable and finitely satisfiable in $A\subseteq M$.  Then there exists a unique extension $p\subseteq q\in S^{\mathcal{L}^{sh}}(\mathbb{U}^*)$ which is definable and finitely satisfiable in $A$.

As a result, if $p$ is generically stable over $M$ then so is $q$.
\end{proposition}
\begin{proof}
Let $M^{sh}\prec L^*$ be an $|M|^+$-saturated extension and $(L^*,M^{sh})\prec (\mathbb{V}^*,\mathbb{U}^{**})$ a saturated extension of cardinality $|\mathbb{U}|$. By uniqueness of saturated models we may assume that $\mathbb{U}^{**}=\mathbb{U}^*$. Let $L,\mathbb{U},\mathbb{V}$ be their reducts to $\mathcal{L}$, respectively.

As $p$ is definable and finitely satisfiable in $A$, there exists an extension $\tilde p\in S(\mathbb{V})$ which is still definable and finitely satisfiable in $A$ (and with the same definition). Seen as a partial type over $\mathbb{V}^*$, $\tilde  p$ is still finitely satisfiable in $A$. Let $\tilde p^{sh}$ be a completion of $\tilde p$ in $S^{\mathcal{L}^{sh}}(\mathbb{V}^*)$ which is still finitely satisfiable in $A$. We will show that $\tilde p^{sh}|\mathbb{U}^*$ is definable over $A$.

For each $c\in N$ choose some $\tilde c\in L$ for which $\tp(c/M)=\tp(\tilde c/M)$. In particular, for each $\mathcal{L}$-formula $\psi(x,y,z)$ and $c\in N$, we have that
\[\psi(M,c)=\psi(M,\tilde c).\]

We first show the following:
\begin{claim}
\[R_{\varphi(x,y,c)}(x,b)\in \tilde p^{sh}|\mathbb{U}^* \iff \varphi(x,b,\tilde c)\in \tilde p \tag{$\star$}.\]
\end{claim}
\begin{claimproof}
Assume that $(\star)$ is not true. Thus there exists $b\in \mathbb{U}^*$ with \[\neg R_{\varphi(x,y,c)}(x,b)\wedge \varphi(x,b,\tilde c)\in \tilde p^{sh}.\] 
Since $\tilde p^{sh}$ is finitely satisfiable in $A$ there exists $m\in A\subseteq M^{sh}$ with 
\[\mathbb{V}^*\models \neg R_{\varphi(x,y,c)}(m,b)\wedge \varphi(m,b,\tilde c),\]
but since we also have
\[(\mathbb{V}^*,\mathbb{U}^*)\models \forall (a_1,a_2)\in \mathbb{U}^*\left(R_{\varphi(x,y,c)}(a_1,a_2)\leftrightarrow \varphi(a_1,a_2,\tilde c)\right),\]
we arrive to a contradiction.
\end{claimproof}

We now show that $\tilde p^{sh}|\mathbb{U}^*$ is definable over $A$. Since $\tilde p$ is definable over $A$, we have $\varphi(x,b,\tilde c)\in \tilde p \iff \mathbb{V}\models \psi(b,\tilde c)$ for some $\mathcal{L}$-formula over $A$. By the choice of $\tilde c$, we have that $\psi (M,\tilde c)=\psi (M,c)$, thus
\[(L^*,M^{sh})\models \forall a\in M^{sh}\left(R_{\psi(x,c)}(a)\leftrightarrow \psi(a,\tilde c)\right).\]
Since $(L^*,M^{sh})\prec (\mathbb{V}^*,\mathbb{U}^*)$, for all $a\in \mathbb{U}^*$
\[\mathbb{U}^*\models R_{\psi(x,c)}(a)\iff \mathbb{V}^*\models \psi(a,\tilde c).\]
Together with the claim this combines to \[R_{\varphi(x,y,c)}(x,b)\in \tilde p^{sh}|\mathbb{U}^*\iff \mathbb{U}^*\models R_{\psi(x,c)}(b).\]

Uniqueness follows from $(\star)$.
\end{proof}

\bibliographystyle{alpha}
\bibliography{genstable-grps2}

\end{document}